\documentclass{article}

\usepackage{microtype}
\usepackage{graphicx}
\usepackage{booktabs} 
\usepackage{enumerate}
\usepackage{epstopdf}
\usepackage{caption}
\usepackage{subcaption}
\usepackage{hyperref}

\usepackage{natbib}
\usepackage{amsmath}
\usepackage{amsthm}
\usepackage{amssymb}
\usepackage{mathtools}
\usepackage{tikz}
\usepackage{xcolor}
\usetikzlibrary{arrows}

\usepackage{algorithm}
\usepackage{algorithmic}
\usepackage{hyperref}
\usepackage{bm}

\newcommand{\argmin}{\mathrm{argmin}}

\newcommand{\poly}{\mathrm{poly}}
\newcommand{\rank}{\mathrm{rank}}

\newcommand{\E}{\mathbb{E}}

\def\R{\mathbb{R}}

\def\cN{\mathcal{N}}

\newcommand{\mat}[1]{{#1}}
\newcommand{\vect}[1]{{#1}}
\newcommand{\norm}[1]{\left\|#1\right\|}

\newcommand{\expect}[1]{\E\left[#1\right]}

\newtheorem{thm}{Theorem}[section]
\newtheorem{lem}{Lemma}[section]
\newtheorem{cor}{Corollary}[section]
\newtheorem{prop}{Proposition}[section]
\newtheorem{asmp}{Assumption}[section]
\newtheorem{defn}{Definition}[section]
\newtheorem{fact}{Fact}[section]

\allowdisplaybreaks

\usepackage[accepted]{aistats2019}

\begin{document}




\twocolumn[

\aistatstitle{Linear Convergence of the Primal-Dual Gradient Method for Convex-Concave Saddle Point Problems without Strong Convexity}

\aistatsauthor{Simon S. Du \And Wei Hu}

\aistatsaddress{Carnegie Mellon University\\\texttt{ssdu@cs.cmu.edu} \And Princeton University\\  \texttt{huwei@cs.princeton.edu}} ]




\begin{abstract}
	\label{sec:abs}
	We consider the convex-concave saddle point problem $\min_{x}\max_{y} f(x)+y^\top A x-g(y)$ where $f$ is smooth and convex and $g$ is smooth and strongly convex.
We prove that if the coupling matrix $A$ has full column rank, the \emph{vanilla} primal-dual gradient method can achieve linear convergence \emph{even if $f$ is not strongly convex}.
Our result generalizes previous work which either requires $f$ and $g$ to be quadratic functions or requires proximal mappings for both $f$ and $g$.
We adopt a novel analysis technique that in each iteration uses a ``ghost" update as a reference, and show that the iterates in  
the primal-dual gradient method converge to this ``ghost" sequence.
Using the same technique we further give an analysis for the primal-dual stochastic variance reduced gradient  method for convex-concave saddle point problems with a finite-sum structure.
\end{abstract}

\section{Introduction}
\label{sec:intro}
We revisit the convex-concave saddle point problems of the form
\begin{equation}
	\min_{\vect{x} \in \mathbb{R}^{d_1}}\max_{\vect{y}\in\mathbb{R}^{d_2}} L(\vect{x},\vect{y}) = f(\vect{x}) + \vect{y}^\top \mat{A}\vect{x}  -g(\vect{y}), \label{eqn:basic_primal_dual}
\end{equation}
where both  $f$ and $g$ are convex functions and $A \in \mathbb{R}^{d_2 \times d_1}$ is a coupling matrix. 
This formulation has a wide range of applications, including supervised learning~\citep{zhang2015stochastic}, unsupervised learning~\citep{xu2005maximum,bach2008convex}, reinforcement learning~\citep{du2017stochastic}, robust optimization~\citep{ben2009robust},  PID control~\citep{hast2013pid}, etc.
See Section~\ref{sec:examples} for some concrete examples.

When the problem dimension is large, the most widely used and sometimes the only scalable methods to solve Problem~\eqref{eqn:basic_primal_dual} are first-order methods.
Arguably the simplest first-order algorithm 
is the \emph{primal-dual gradient method} (Algorithm~\ref{algo:pdg}), a natural generalization of the gradient descent algorithm,
which simultaneously performs gradient descent on the primal variable $x$ and gradient ascent on the dual variable $y$.
\begin{algorithm}[tb]
	\renewcommand{\algorithmicrequire}{\textbf{Inputs:}}
	\renewcommand{\algorithmicensure}{\textbf{Outputs:}}
	\caption{Primal-Dual Gradient Method}
	\label{algo:pdg}
	\begin{algorithmic}[1]
		\REQUIRE initial points $x_0 \in \R^{d_1}, y_0\in \R^{d_2}$, 
		step sizes $\eta_1, \eta_2 >0$ 
		\FOR{$t=0, 1,\ldots$ }
\STATE $\vect{x}_{t+1} = \vect{x}_t - \eta_1 \nabla_x L(x_t, y_t)$\\ \qquad \  $  = x_t - \eta_1\left(\nabla f(x_t) +A^\top y_t\right)$
\STATE $\vect{y}_{t+1} = \vect{y}_t + \eta_2 \nabla_y L(x_t, y_t)$\\ \qquad \  $  = y_t + \eta_2\left(Ax_t - \nabla g(y_t)\right)$
		\ENDFOR
	\end{algorithmic}
\end{algorithm}

There has been extensive research on analyzing the convergence rate of Algorithm~\ref{algo:pdg} and its variants.
It is known that 
if both $f$ and $g$ are strongly convex and admit efficient proximal mappings, then the \emph{proximal} primal-dual gradient method converges to the optimal solution at a linear rate~\citep{bauschke2011convex,palaniappan2016stochastic,chen1997convergence}, i.e., it only requires $O\left(\log\frac{1}{\epsilon}\right)$ iterations to obtain a solution that is  $\epsilon$-close to the optimum.

In many applications, however, we only have strong convexity in $g$ but \emph{no} strong convexity in $f$.
This motivates the following question:

\textbf{Does the primal-dual gradient method converge linearly to the optimal solution if $f$ is not strongly convex?}

Intuitively, a linear convergence rate is plausible.
Consider the corresponding primal problem of~\eqref{eqn:basic_primal_dual}:
\begin{align}
	\min_{x \in \mathbb{R}^{d_1}} P(x) = g^*(Ax) + f(x), \label{eqn:primal_problem}
\end{align}
where $g^*$ is the conjugate function of $g$.
Because $g$ is smooth and strongly convex, as long as $\mat{A}$ has full column rank, Problem~\eqref{eqn:primal_problem} has a smooth and strongly convex objective 
and thus vanilla gradient descent achieves linear convergence.
Therefore, one might expect a linearly convergent first-order algorithm for Problem~\eqref{eqn:basic_primal_dual} as well.
However, whether the vanilla primal-dual gradient method (Algorithm~\ref{algo:pdg}) has linear convergence turns out to be a nontrivial question.

Two recent results verified this conceptual experiment with additional assumptions:
\citet{du2017stochastic} required both $f$ and $g$ to be quadratic functions, and \citet{wang2017exploiting} required both $f$ and $g$ to have efficient proximal mappings and uses a \emph{proximal} primal-dual gradient method.
In this paper, we give an affirmative answer to this question with minimal assumptions.
Our main contributions are summarized below.

\subsection{Our Contributions}
\paragraph{Linear Convergence of the Primal-Dual Gradient Method.}
We show that as long as $f$ and $g$ are smooth, $f$ is convex, $g$ is strongly convex and the coupling matrix $\mat{A}$ has full column rank, Algorithm~\ref{algo:pdg} converges to the optimal solution at a linear rate.
See Section~\ref{sec:pdbg} for a precise statement of our result.
This result significantly generalizes previous ones which rely on stronger assumptions.
Note that all the assumptions are necessary for linear convergence: without any of them, the primal problem~\eqref{eqn:primal_problem} requires at least $\poly(\frac{1}{\epsilon})$ iterations to obtain an $\epsilon$-close solution~\citep{nesterov2013introductory}, so there is no hope of linear convergence for Problem~\eqref{eqn:basic_primal_dual}.

\paragraph{New Analysis Technique.}
To analyze the convergence of an optimization algorithm, a common way is to construct a potential function (also called Lyapunov function in the literature) which decreases after each iteration.
For example, for the primal problem~\eqref{eqn:primal_problem}, a natural potential function is $\norm{x_t - x^*}$, the distance between the current iterate and the optimal solution.
However, for the primal-dual gradient method, it is difficult to show similar potential functions like $\norm{x_t - x^*} + \norm{y_t - y^*}$ decrease because the two sequences, $\left\{x_t\right\}_{t=0}^\infty$ and $\left\{y_t\right\}_{t=0}^\infty$, are related to each other.

In this paper, we develop a novel method for analyzing the convergence rate of the primal-dual gradient method.
The key idea is to consider a ``ghost" sequence. For example, in our setting, the ``ghost" sequence comes from a gradient descent step for Problem~\eqref{eqn:primal_problem}. 
Then we relate the sequence generated by Algorithm~\ref{algo:pdg} to this ``ghost" sequence and show they are close in a certain way.
See Section~\ref{sec:pdbg} for details.
We believe this technique is applicable to other problems where we need to analyze multiple sequences.

\paragraph{Extension to Primal-Dual Stochastic Variance Reduced Gradient Method.}
Many optimization problems in machine learning have a finite-sum structure, and randomized algorithms have been proposed to exploit this structure and to speed up the convergence.
There has been extensive research in recent years on developing more efficient stochastic algorithms in such setting \citep{roux2012stochastic,johnson2013accelerating,defazio2014saga,xiao2014proximal,shalev2013stochastic,richtarik2014iteration,lin2015universal,zhang2015stochastic,allen2017katyusha}.
Among them, the
stochastic variance reduced gradient (SVRG) algorithm~\citep{johnson2013accelerating} is a popular one with computational complexity $O\left((n+\kappa)d\log\frac1\epsilon\right)$ for smooth and strongly convex objectives, where $n$ is the number of component functions, $d$ is the dimension of the variable, and $\kappa$ is a condition number that only depends on problem-dependent parameters like smoothness and strong convexity but not $n$.
Variants of SVRG for saddle point problems have been recently studied by~\citet{palaniappan2016stochastic,wang2017exploiting,du2017stochastic} and can achieve similar $O\left((n+\kappa)d\log\frac1\epsilon\right)$ running time.\footnote{$\kappa$ may be different in the primal and the primal-dual settings.}
However, these results all require additional assumptions.
In this paper, we use our analysis technique developed for Algorithm~\ref{algo:pdg} to show that the primal-dual SVRG method also admits $O\left((n+\kappa)d\log\frac1\epsilon\right)$ type computational complexity.


\subsection{Motivating Examples}
\label{sec:examples}
In this subsection we list some machine learning applications that naturally lead to  convex-concave saddle point problems.

\paragraph{Reinforcement Learning.}
For policy evaluation task in reinforcement learning, we have data $\left\{\left(s_t,r_t,s_{t+1}\right)\right\}_{t=1}^n$ generated by a policy $\pi$ where $s_t$ is the state at the $t$-th time step, $r_t$ is the reward and $s_{t+1}$ is the state at the $(t+1)$-th step.
We also have a discount factor $0< \gamma < 1$ and a feature function $\phi(\cdot)$ which maps a state to a feature vector. 
Our goal is to learn a linear value function $V^\pi\left(s\right) \approx x^\top \phi\left(s\right)$ which represents the long term expected reward starting from state $s$ using the policy $\pi$.
A common way to estimate $x$ is to minimize the empirical mean squared projected Bellman error (MSPBE):\begin{align}
\min_x \left(Ax-b\right)^\top C^{-1}\left(Ax-b\right), \label{eqn:mspbe}
\end{align}where $A = \sum_{t=1}^{n}\phi(s_t)\left(\phi(s_t)-\gamma\phi(s_{t+1})\right)^\top$, $b=\sum_{t=1}^{n}r_t\phi(s_t)$ and $C = \sum_{t=1}^{n}\phi(s_t)\phi(s_t)^\top$. 
Note that directly using gradient descent to solve problem~\eqref{eqn:mspbe} is expensive because we need to invert a matrix $\mat{C}$.
\cite{du2017stochastic} considered the equivalent saddle point formulation:
\begin{align*}
\min_x\max_y L(x,y)= -y^\top A x -\frac{1}{2}y^\top \mat{C}y  + b^\top y.
\end{align*}
The gradient of $L$ can be computed more efficiently than the original formulation \eqref{eqn:mspbe}, and $L$ has a finite-sum structure.

\paragraph{Empirical Risk Minimization.}
Consider the classical supervised learning problem of learning a linear predictor $x\in \R^d$ given $n$ data points $(a_i, b_i) \in \R^d \times \R$.
Denote by $A \in \R^{n\times d}$ the data matrix whose $i$-th row is $a_i^\top$.
Then the empirical risk minimization (ERM) problem amounts to solving
\begin{align*}
\min_{x\in \R^d} \ell(Ax) + f(x),
\end{align*}
where $\ell$ is induced by some loss function and $f$ is a regularizer; both $f$ and $\ell$ are convex functions.
Equivalently, we can solve the dual problem $\max_{y\in\R^n} \left\{ -\ell^*(y)-f^*(-A^\top y) \right\}$ or the saddle point problem
$\min_{x\in\R^d}\max_{y\in\R^n} \left\{ y^\top A x - \ell^*(y) + f(x)\right\}$.
The saddle point formulation is favorable in many scenarios, e.g., when such formulation admits a  finite-sum structure~\citep{zhang2015stochastic,wang2017exploiting}, reduces communication complexity in the distributed setting~\citep{xiao2017dscovr} or exploits sparsity structure~\citep{lei2017doubly}.

\paragraph{Robust Optimization.}
The robust optimization framework~\citep{ben2009robust} aims at minimizing an objective function with uncertain data, which naturally leads to
a saddle point problem, often with the following form:\begin{align}
	\min_x \max_y \mathbb{E}_{\xi\sim P(y)}\left[f(x,\xi)\right], \label{eqn:dro}
\end{align} where $f$ is some loss function we want to minimize and the distribution of the data is parametrized by $P(y)$.
For certain special cases~\citep{liu2017primal}, Problem~\eqref{eqn:dro} has the bilinear form as in~\eqref{eqn:basic_primal_dual}.

\subsection{Comparison with Previous Results}
\label{sec:rel}
\begin{table*}[tb]
	\centering
	\resizebox{2\columnwidth}{!}{%

	\renewcommand{\arraystretch}{1.5}
	\begin{tabular}{ |c|c|c|c|c|c|c|}
		\hline
		{\bf Paper} & $f$ smooth &$f$ s.c.& $g$ smooth & $g$ s.c. 
		 & $A$ full column rank & Other Assumptions\\ 
		\hline
		\citep{chen1997convergence} & \textbackslash & Yes& \textbackslash & Yes  & No & Prox maps for $f$ and $g$\\
		\hline
		\citep{du2017stochastic} &  Yes &  No & Yes & Yes  & Yes & $f$ and $g$ are quadratic\\
		\hline
		\citep{wang2017exploiting} &  \textbackslash &  No & \textbackslash & Yes  & Yes & Prox maps for $f$ and $g$ \\
		\hline
		Folklore &  Yes &  Yes & Yes & Yes  & No & No \\
		\hline
		\bf This Paper&  Yes &  No & Yes & Yes  & Yes & No\\
		\hline
	\end{tabular}
	}
	\caption{Comparisons of assumptions that lead to the linear convergence of  primal-dual gradient method for solving Problem~\eqref{eqn:basic_primal_dual}.
		When we have proximal mappings for $f$ and $g$, we do not need their smoothness.		\label{tab:complexity}
	}
\end{table*}

There have been many attempts to analyze the primal-dual gradient method or its variants.
In particular,~\citet{chen1997convergence,chambolle2011first,palaniappan2016stochastic} showed that if both $f$ and $g$ are strongly convex and have efficient proximal mappings, then the proximal primal-dual gradient method achieves a linear convergence rate.\footnote{\citet{chen1997convergence,palaniappan2016stochastic} considered a more general formulation than Problem~\eqref{eqn:basic_primal_dual}. Here we specialize in the bi-linear saddle point problem.}
In fact, even without proximal mappings, as long as both $f$ and $g$ are smooth and strongly convex, Algorithm~\ref{algo:pdg} achieves a linear convergence rate.
In Appendix~\ref{sec:both_strong_proof} we give a simple proof of this fact.

Two recent papers show that it is possible to achieve linear convergence even without strong convexity in $f$.
The key is the additional assumption that $A$ has full column rank, 
which helps ``transfer" $g$'s strong convexity to $f$.
\cite{du2017stochastic} considered the case when both $f$ and $g$ are quadratic functions, i.e., when Problem~\eqref{eqn:basic_primal_dual} has the following special form:
\begin{align*}
 L(x,y) = x^\top Bx +b^\top x + y^\top Ax - y^\top C y+c^\top y.
\end{align*}
Note that $B$ does not have to be positive definite (but $C$ has to be), and thus strong convexity is not necessary in the primal variable.
Their analysis is based on writing the gradient updates as a linear dynamic system (c.f. Equation~(41) in~\citep{du2017stochastic}):
\begin{align}
\begin{bmatrix}
x_{t+1} - x^*\\
\sqrt{\frac{\eta_1}{\eta_2}}\left(y_{t+1}-y^*\right)
\end{bmatrix} 
=  \left(I-G\right)
\begin{bmatrix}
x_{t} - x^*\\
\sqrt{\frac{\eta_1}{\eta_2}}\left(y_{t}-y^*\right)
\end{bmatrix}, \label{eqn:linear_dynamical_system}
\end{align} where $G$ is some fixed matrix that depends on $A,B,C$ and step sizes.
Next, it suffices to bound the spectral norm of $G$ (which can be made strictly less than $1$)
to show that $\left(
x_{t} - x^*,
\sqrt{\frac{\eta_1}{\eta_2}}\left(y_{t}-y^*\right)\right)
$ converges to $(0,0)$ at a linear rate.
However, it is difficult to generalize this approach to general saddle point problem~\eqref{eqn:basic_primal_dual} since only when $f$ and $g$ are quadratic do we have the linear form~\eqref{eqn:linear_dynamical_system}.

\cite{wang2017exploiting} considered the proximal primal-dual gradient method.
They construct a potential function (c.f. Page 15 in~\citep{wang2017exploiting})  and show it decreases at a linear rate.
However, this potential function heavily relies on the proximal mappings so it is difficult to use this technique to analyze Algorithm~\ref{algo:pdg}.

In Table~\ref{tab:complexity}, we summarize different assumptions sufficient for linear convergence used in different papers.

\subsection{Paper Organization}
The rest of the paper is organized as follows.
We give necessary definitions in Section~\ref{sec:pre}.
In Section~\ref{sec:pdbg}, we present our main result for the primal-dual gradient method and its proof.
In Section~\ref{sec:svrg}, we extend our analysis to the primal-dual stochastic variance reduced gradient method.
In Section~\ref{sec:exp}, we use some preliminary experiments to verify our theory.
We conclude in Section~\ref{sec:con} and put omitted proofs in the appendix.

\section{Preliminaries}
\label{sec:pre}
Let $\norm{\cdot}$ denote the Euclidean ($L_2$) norm of a vector, 
and let $\langle \cdot, \cdot \rangle$ denote the standard Euclidean inner product between two vectors.
For a matrix $\mat{A} \in \R^{m\times n}$, let $\sigma_i(\mat{A})$ be its $i$-th largest singular value, and let $\sigma_{\max}\left(\mat{A}\right) := \sigma_1(\mat{A})$ and $\sigma_{\min}(\mat{A}) := \sigma_{\min\{m, n\}} (\mat{A})$ be the largest and the smallest singular values of $A$, respectively.
For a function $f$, we use $\nabla f$ to denote its gradient.
Denote $[n] := \{1, 2, \ldots, n\}$.
Let $I_d$ be the identity matrix in $\R^{d\times d}$.

The smoothness and the strong convexity of a function are defined as follows:
\begin{defn}
	For a differentiable function $\phi:\R^d \to \R$, we say
	\begin{itemize}
		\item $\phi$ is $\beta$-smooth if 
		$\norm{\nabla \phi(\vect{u})-\nabla \phi(\vect{v})} \le \beta \norm{\vect{u}-\vect{v}}$ for all $u, v \in \R^d$;
		\item $\phi$ is $\alpha$-strongly convex if $\phi(v) \ge \phi(u) + \langle \nabla \phi(u), v-u \rangle + \frac{\alpha}{2}\|u-v\|^2$ for all $u, v \in \R^d$.
	\end{itemize}
\end{defn}

We also need the definition of conjugate function:
\begin{defn}
	The conjugate of a function $\phi: \R^d \to \R$ is defined as
	\[
	\phi^*(y) := \sup_{x\in \R^d} \left\{ \langle x, y \rangle - \phi(x) \right\}, \qquad \forall y\in\R^d.
	\]
\end{defn}

It is well-known that if $\phi$ is closed and convex, then $\phi^{**} = \phi$.
If $\phi$ is smooth and strongly convex, its conjugate $\phi^*$ has the following properties:

\begin{fact} \label{fact:conjugate}
	If $\phi:\R^d\to \R$ is $\beta$-smooth and $\alpha$-strongly convex ($\beta\ge\alpha>0$), then 
	\begin{enumerate}[(i)]
		\item (\citep{kakadeduality})	$\phi^*:\R^d\to\R$ is $\frac{1}{\alpha}$-smooth and $\frac{1}{\beta}$-strongly convex.
		\item (\citep{rockafellar1970convex}) The gradient mappings $\nabla \phi$ and $\nabla \phi^*$ are inverse of each other.
	\end{enumerate}
\end{fact}

\section{Linear Convergence of the Primal-Dual Gradient Method}
\label{sec:pdbg}
In this section we show the linear convergence of Algorithm~\ref{algo:pdg} on Problem~\eqref{eqn:basic_primal_dual} under the following assumptions:

\begin{asmp}\label{asmp:f}
	$f$ is convex and $\rho$-smooth ($\rho\ge0$). 
\end{asmp}
\begin{asmp}\label{asmp:g}
	$g$ is $\beta$-smooth and $\alpha$-strongly convex ($\beta \ge \alpha >0$). 
\end{asmp}
\begin{asmp}\label{asmp:A}
	The matrix $A \in \R^{d_2 \times d_1}$ satisfies $\rank(A) = d_1$.
\end{asmp}

While the first two assumptions on $f$ and $g$ are standard in convex optimization literature, the third one is important for ensuring linear convergence of Problem~\eqref{eqn:basic_primal_dual}.
Note, for example, that if $\mat{A}$ is the all-zero matrix, then there is no interaction between $\vect{x}$ and $\vect{y}$, and to solve the convex optimization problem on $\vect{x}$ we need at least $\Omega\left(\frac{1}{\sqrt{\epsilon}}\right)$ iterations~\citep{nesterov2013introductory} instead of $O\left( \log\frac{1}{\epsilon} \right)$.

Denote by $(x^*, y^*) \in \R^{d_1} \times \R^{d_2}$ the optimal solution to Problem~\eqref{eqn:basic_primal_dual}.
For simplicity, we let $\sigma_{\max} := \sigma_{\max}(A)$ and $\sigma_{\min} := \sigma_{\min}(A)$.

Recall 
the first-order optimality condition:
\begin{equation} \label{eqn:opt-condition}
\begin{cases}
\nabla_x L(x^*, y^*) = \nabla f(x^*) + A^\top y^* = 0, \\
\nabla_y L(x^*, y^*) = -\nabla g(y^*) + Ax^* = 0.
\end{cases}
\end{equation}

\begin{thm}\label{thm:main}
	In the setting of Algorithm~\ref{algo:pdg}, define $a_t := \| x_{t} - x^* \|$ and $b_t := \left\|  y_t -  \nabla g^* (Ax_t)  \right\|$. Let $\lambda := \frac{2 \beta \sigma_{\max}\cdot \left(\rho+\frac{\sigma_{\max}^2}{\alpha}\right)}{\alpha \sigma_{\min}^2}$ and $P_t := \lambda a_t + b_t$.
If we choose $\eta_1 = \frac{\alpha}{(\alpha+\beta) \left( \frac{\sigma_{\max}^2}{\alpha} + \lambda\sigma_{\max} \right)}$ and $\eta_2 = \frac{2}{\alpha+\beta}$,
then we have
	\begin{align*}
	P_{t+1} 
	\le \left( 1 -  C \cdot \frac{\alpha^2 \sigma_{\min}^4}{\beta^3 \sigma_{\max}^2 \cdot \left(\rho+ \frac{\sigma_{\max}^2}{\alpha} \right)} \right) P_t
	\end{align*} for some absolute constant $C > 0$.
\end{thm}

In this theorem, we use $P_t = \lambda a_t + b_t$ as the potential function and show that this function shrinks at a geometric rate.
Note that from~\eqref{eqn:opt-condition} and Fact~\ref{fact:conjugate} (ii) we have $y^* = (\nabla g)^{-1} (Ax^*) = \nabla g^*(Ax^*)$.
Then we have upper bounds 
$\norm{x_t - x^*} = a_t \le \frac1\lambda P_t$
and
$\norm{y_t - y^*} \le \norm{y_t - \nabla g^* (Ax_t)} + \norm{\nabla g^* (Ax_t) - y^*} = b_t + \norm{\nabla g^* (Ax_t) - \nabla g^* (Ax^*)} \le b_t + \frac{\sigma_{\max}}{\alpha} a_t \le \max\left\{1, \frac{\sigma_{\max}}{\alpha\lambda}\right\}P_t$,
which imply that if $P_t$ is small then $(x_t, y_t)$ will be close to the optimal solution $(x^*, y^*)$.
Therefore a direct corollary of Theorem~\ref{thm:main} is: 
\begin{cor}\label{cor:distance_to_opt}
	For any $\epsilon>0$, after $O^*\left(\log\frac{P_0}{\epsilon} \right)$ iterations, we have 
	$\norm{x_t-x^*} \le \epsilon$ and $\norm{y_t-y^*}\le \epsilon$, where $O^*(\cdot)$ hides polynomial factors in $\beta, 1/\alpha, \sigma_{\max}, 1/\sigma_{\min}$ and $\rho$.
\end{cor}

We remark that our theorem suggests that step sizes depend on problem parameters which may be unknown.
In practice, we may try to use a small amount of data to estimate them first or use the adaptive tuning heuristic introduced in~\citep{wang2017exploiting}.


\subsection{Proof of Theorem~\ref{thm:main}} \label{sec:proof-main-thm}
Now we present the proof of Theorem~\ref{thm:main}.

First recall the standard linear convergence guarantee of gradient descent on a smooth and strongly convex objective. See Theorem 3.12 in \citep{bubeck2015convex} for a proof.
\begin{lem} \label{lem:gd-converge}
	Suppose $\phi:\R^d\to\R$ is $\gamma$-smooth and $\delta$-strongly convex, and let $\bar x := \argmin_{x\in\R^d}\phi(x)$.
	For any $0<\eta \le \frac{2}{\gamma+\delta}$, $x \in \R^d$, letting $\tilde{x} = x - \eta \nabla \phi(x)$, we have
	\[
	\norm{\tilde{x} - \bar x} \le (1-\delta\eta) \norm{x - \bar x}.
	\]
\end{lem}

\paragraph{Step 1: Bounding the Decrease of $\norm{x_{t} - x^*}$ via a One-Step ``Ghost" Algorithm.\footnote{$\norm{x_{t} - x^*}$ may not decrease as $t$ increases. 
Here what we mean is to upper bound $\norm{x_{t+1}-x^*}$ using $\norm{x_{t} - x^*}$ and an error term.}}
Our technique is to consider the following one-step ``ghost" algorithm for the primal variable, which corresponds to a gradient descent step for the primal problem~\eqref{eqn:primal_problem}.
We define an auxiliary variable $\tilde x_{t+1}$: given $x_t$, let
\begin{align}\label{eqn:primal-gd}
	\tilde x_{t+1} :=  x_t - \eta_1 \left( \nabla f( x_t) + \nabla h(x_t) \right).
\end{align}
where $h(x) := g^*(Ax)$.
Note that $\tilde x_{t+1}$ is defined only for the purpose of the proof.
Our main idea is to use this ``ghost" algorithm as a reference and bound the distance between the primal-dual gradient iterate $x_{t+1}$ and this ``ghost'' variable $\tilde x_{t+1}$.
We first prove with this ``ghost" algorithm, the distance between the primal variable and the optimum $x^*$ decreases at a geometric rate.

\begin{prop}\label{prop:ghost-linear-converge}	
If $\eta_1 \le \frac{2}{\rho + \sigma_{\max}^2/\alpha + \sigma_{\min}^2/\beta}$, then \[
\norm{\tilde{x}_{t+1}-x^*} \le \left(1-\frac{\sigma_{\min}^2}{\beta}\eta_1\right)\norm{x_{t}-x^*}.
\]
\end{prop}
\begin{proof}
	Since \eqref{eqn:primal-gd} is a gradient descent step for the primal problem~\eqref{eqn:primal_problem} whose objective is $P(x) = h(x)+f(x)$ where $h(x) = g^*(Ax)$, it suffices to show that $P$ is smooth and strongly convex in order to apply Lemma~\ref{lem:gd-converge}.
	Note that $g^*$ is $\frac{1}{\alpha}$-smooth and $\frac{1}{\beta}$-strongly convex according to Fact~\ref{fact:conjugate}.
	
	We have $\nabla h(x) = A^\top \nabla g^*(Ax)$. Then for any $x, x' \in \R^d$ we have
	\begin{align*}
	&\norm{\nabla P(x) - \nabla P(x')} \\
	\le\, & \norm{\nabla f(x) - \nabla f(x')} + \norm{A^\top \nabla g^*(Ax) - A^\top \nabla g^*(Ax')} \\
	\le\, & \rho \norm{x-x'} + \sigma_{\max} \norm{\nabla g^*(Ax) - \nabla g^*(Ax')} \\
	\le\, & \rho \norm{x-x'} + \frac{\sigma_{\max}}{\alpha} \norm{Ax -  Ax'} \\
	\le\, & \rho \norm{x-x'} + \frac{\sigma_{\max}^2}{\alpha} \norm{x -  x'} \\
	=\, & (\rho + \sigma_{\max}^2/\alpha) \norm{x-x'},
	\end{align*}
	where we have used the $\rho$-smoothness of $f$, the $\frac{1}{\alpha}$-smoothness of $g^*$, and the bound on $\sigma_{\max}(A)$.
	Therefore $P$ is $(\rho + \sigma_{\max}^2/\alpha)$-smooth.
	
	On the other hand, for any $x, x' \in \R^d$ we have
	\begin{align*}
	&P(x') - P(x) \\
	=\, & f(x') - f(x) + g^*(Ax') - g^*(Ax) \\
	\ge\, & \langle \nabla f(x), x'-x \rangle + \langle \nabla g^*(Ax), Ax'-Ax \rangle \\ &+ \frac{1/\beta}{2} \norm{Ax'-Ax}^2 \\
	=\, & \langle \nabla f(x) + A^\top \nabla g^*(Ax), x'-x \rangle + \frac{1}{2\beta} \norm{Ax'-Ax}^2 \\
	\ge\, & \langle \nabla P(x) , x'-x \rangle + \frac{1}{2\beta} \sigma_{\min}^2 \norm{x'-x}^2,
	\end{align*}
	where we have used the convexity of $f$, the $\frac1\beta$-strong convexity of $g^*$, and that $A$ has full column rank.
	Therefore $P$ is $\sigma_{\min}^2/\beta$-strongly convex.
	
	With the smoothness and the strong convexity of $P$,
	the proof is completed by applying Lemma~\ref{lem:gd-converge}.
\end{proof}

Proposition~\ref{prop:ghost-linear-converge} suggests that if we use the ``ghost" algorithm \eqref{eqn:primal-gd}, we have the desired linear convergence property.
The following proposition gives an upper bound on $\norm{x_{t+1} - x^*}$ by bounding
the distance between $x_{t+1}$ and $\tilde{x}_{t+1}$.

\begin{prop} \label{prop:primal}
	If $\eta_1 \le \frac{2}{\rho + \sigma_{\max}^2/\alpha + \sigma_{\min}^2/\beta}$, 
	then
	\begin{equation} \label{eqn:primal-decrease}
	\begin{aligned}
	\left\| x_{t+1} - x^* \right\|
	\le & \left( 1-   \frac{\sigma_{\min}^2}{\beta}\eta_1 \right) \left\| x_t - x^* \right\| \\&+ 
	\sigma_{\max} \eta_1 \left\|  y_t - \nabla g^* (Ax_t)  \right\|.
	\end{aligned}
	\end{equation}
\end{prop}
\begin{proof}
	We have $\tilde x_{t+1} - x_{t+1} = \eta_1 A^\top (y_t - \nabla g^*(A x_t))$, which implies
	\[
	\norm{ \tilde x_{t+1} - x_{t+1} } \le \eta_1 \sigma_{\max} \norm{y_t - \nabla g^*(A x_t)}.
	\]
	Then the proposition follows by applying the triangle inequality and Proposition~\ref{prop:ghost-linear-converge}.
\end{proof}


\paragraph{Step 2: Bounding the Decrease of $\norm{y_t - \nabla g^*(A x_t)}$.}
One may want to show the decrease of $\norm{y_t - y^*}$ similarly using a ``ghost'' update for the dual variable.
However, the objective function in the dual problem $\max_y \left\{ -g(y)-f^*(-A^\top y) \right\}$ might be non-smooth, which means we cannot obtain a result similar to Proposition~\ref{prop:ghost-linear-converge}.
Instead, we show that $\norm{y_t - \nabla g^*(A x_t)}$ decreases geometrically up to an error term.


\begin{prop}\label{prop:xt_xt1}
	We have
	\begin{align*}
	\norm{x_{t+1}-x_t} \le& \left( \rho + \frac{\sigma_{\max}^2}{\alpha} \right) \eta_1 \norm{x_t-x^*} \\ &+ \sigma_{\max}\eta_1 \norm{y_t - \nabla g^*(Ax_t)}.
	\end{align*}
\end{prop}
\begin{proof}
	Using the gradient update formula of the primal variable, we have
	\begin{equation} \label{eqn:x_t-x_{t+1}}
	\begin{aligned}
	&\frac{1}{\eta_1}\|x_{t+1} - x_t\|  
	=  \left\| \nabla f(x_{t}) + A^\top y_t \right\| \\
	\le\,&     \left\| \nabla f(x_{t}) + A^\top \nabla g^*(Ax_t) \right\| + \left\| A^\top (y_t - \nabla g^*(Ax_t)  ) \right\|   \\
	\le\, &  \left\| \nabla f(x_{t}) + A^\top \nabla g^*(Ax_t) \right\| + \sigma_{\max} \norm{y_t - \nabla g^*(Ax_t)} . 
	\end{aligned}
	\end{equation}
	
	Recall that the primal objective function $P(x) = f(x) + g^*(Ax)$ is $(\rho+\sigma_{\max}^2/\alpha)$-smooth (see the proof of Proposition~\ref{prop:ghost-linear-converge}).
	So we have
	\begin{align*}
	&\norm{\nabla f(x_{t}) + A^\top \nabla g^* (Ax_t)}
	= \norm{\nabla P(x_t)} \\
	=\,& \norm{\nabla P(x_t) - \nabla P(x^*)}
	\le (\rho+\sigma_{\max}^2/\alpha) \norm{x_t-x^*}.
	\end{align*}
	Plugging this back to~\eqref{eqn:x_t-x_{t+1}} we obtain the desired result.
\end{proof}

\begin{prop} \label{prop:dual}
	If $\eta_2 \le \frac{2}{\alpha+\beta}$, then
	\begin{align*}
	&\left\| y_{t+1} - \nabla g^*(Ax_{t+1}) \right\| \\
	\le\, & \left(1-\alpha\eta_2 + \frac{\sigma^2_{\max}}{\alpha} \eta_1\right) \norm{ y_{t} - \nabla g^*(Ax_{t}) } \\ 
	&+ \frac{\sigma_{\max}}{\alpha} \left( \rho + \frac{\sigma_{\max}^2}{\alpha} \right) \eta_1 \norm{x_t-x^*}.
	\end{align*}
\end{prop}
\begin{proof}
	For fixed $x_t$, the update rule $y_{t+1} = y_t - \eta_2 (\nabla g(y_t) - Ax_t)$ is a gradient descent step for the objective function $\tilde g(y) := g(y) - y^\top Ax_t$ which is also $\beta$-smooth and $\alpha$-strongly convex. By the optimality condition, the minimizer $\tilde{y}^* = \argmin_{y\in\R^d} \tilde g(y)$ satisfies $\nabla g(\tilde y^*) = Ax_t$, i.e., $\tilde y^* = \nabla g^*(Ax_t)$.
	Then from Lemma~\ref{lem:gd-converge} we know that
	\begin{equation} \label{eqn:dual-decrease-1}
	\norm{y_{t+1} - \nabla g^*(Ax_t)} \le (1-\alpha\eta_2) \norm{y_t - \nabla g^*(Ax_t)}.
	\end{equation}
	
	Since we want to upper bound $\norm{y_{t+1} - \nabla g^*(Ax_{t+1})}$, we need to take into account the difference between $x_{t+1}$ and $x_t$.
	We prove an upper bound on $\norm{x_{t+1}-x_t}$ in Proposition~\ref{prop:xt_xt1}.
	Using Proposition~\ref{prop:xt_xt1} and \eqref{eqn:dual-decrease-1}, we have
	\begin{align*}
		&\norm{ y_{t+1} - \nabla g^*(Ax_{t+1}) } \\
		\le\, & \norm{ y_{t+1} - \nabla g^*(Ax_{t}) } + \norm{\nabla g^*(Ax_{t+1}) - \nabla g^*(Ax_{t})} \\
		\le\, & \norm{ y_{t+1} - \nabla g^*(Ax_{t}) } + \frac{\sigma_{\max}}{\alpha} \norm{ x_{t+1} - x_{t}} \\
		\le\, & (1-\alpha\eta_2) \norm{ y_{t} - \nabla g^*(Ax_{t}) } \\ 
		&+ \frac{\sigma_{\max}}{\alpha} \left( \rho + \frac{\sigma_{\max}^2}{\alpha} \right) \eta_1 \norm{x_t-x^*} \\
		&+ \frac{\sigma^2_{\max}}{\alpha} \eta_1 \norm{ y_{t} - \nabla g^*(Ax_{t}) }. \qedhere
	\end{align*}
\end{proof}


Note that the upper bound on $\norm{x_{t+1}-x_t}$ given in 
Proposition~\ref{prop:xt_xt1}  is proportional to $\eta_1$, not to $\eta_2$.
This allows us to choose a relatively small $\eta_1$ to ensure that the factor $1-\alpha\eta_2 + \frac{\sigma_{\max}^2}{\alpha}\eta_1$ in Proposition~\ref{prop:dual} is indeed less than $1$, i.e., $\norm{y_t - \nabla g^*(Ax_t)}$ is approximately decreasing.

\paragraph{Step 3: Putting Things Together.} 
Now we are ready to finish the proof of Theorem~\ref{thm:main}.
From Propositions~\ref{prop:primal} and~\ref{prop:dual} we have
\begin{equation} \label{eqn:main-ineq-1}
a_{t+1}
\le \left( 1-  \frac{\sigma_{\min}^2}{\beta} \eta_1 \right) a_t + \sigma_{\max} \eta_1 b_t,
\end{equation}
\begin{equation} \label{eqn:main-ineq-2}
\begin{aligned}
b_{t+1} \le\, &\frac{\sigma_{\max}}{\alpha} \left( \rho + \frac{\sigma_{\max}^2}{\alpha} \right) \eta_1 a_t  \\
&+\left( 1-\alpha \eta_2 + \frac{\sigma_{\max}^2}{\alpha}\eta_1 \right) b_t .
\end{aligned}
\end{equation}

To prove the convergence of sequences $\left\{a_t\right\}$ and $\left\{b_t\right\}$ to $0$, we consider a linear combination $P_t = \lambda a_t + b_t$ with a free parameter $\lambda>0$ to be determined.
Combining \eqref{eqn:main-ineq-1} and \eqref{eqn:main-ineq-2},
with some routine calculations, we can show that our choices of $\lambda$, $\eta_1$ and $\eta_2$ given in Theorem~\ref{thm:main} can ensure $P_{t+1} \le cP_t$ for some $0<c<1$, as desired. 
We give the remaining details in Appendix~\ref{sec:proof-main}. 
	

\section{Extension to Primal-Dual SVRG}
\label{sec:svrg}

In this section we consider the case where the saddle point problem \eqref{eqn:basic_primal_dual} admits a finite-sum structure:\footnote{For ease of presentation we assume $f$, $g$ and $K$ can be split into $n$ terms. 
	It is not hard to generalize our analysis to the case where $f$, $g$ and $A$ can be split into different numbers of terms.}
\begin{align} \label{eqn:finite_sum}
\min_{\vect{x}\in \R^{d_1}} \max_{y\in \R^{d_2}}L(\vect{x},\vect{y}) = \frac{1}{n} \sum_{i=1}^n L_i(x, y),
\end{align}
where
$L_i(x, y) :=  f_i(\vect{x}) + \vect{y}^\top A_i \vect{x}  - g_i(\vect{y})$.
Optimization problems with finite-sum structure are ubiquitous in machine learning, because loss functions can often be written as a sum of individual loss terms corresponding to individual observations.

In this section, we make the following assumptions:
\begin{asmp}\label{asmp:f-svrg}
	Each $f_i$ is $\rho$-smooth ($\rho\ge0$), and $f = \frac1n \sum_{i=1}^n f_i$ is convex.
\end{asmp}
\begin{asmp}\label{asmp:g-svrg}
	Each $g_i$ is $\beta$-smooth, and $g = \frac1n \sum_{i=1}^n g_i$ is $\alpha$-strongly convex ($\beta \ge \alpha >0$). 
\end{asmp}
\begin{asmp}\label{asmp:A-svrg}
	Each $A_i$ satisfies $\sigma_{\max}(A_i) \le M$, and $A = \frac1n \sum_{i=1}^n A_i$ has rank $d_1$.
\end{asmp}

Note that we only require component functions $f_i$ and $g_i$ to be smooth; they are not necessarily convex.
However, the overall objective function $L(x, y) = f(x) +y^\top Ax - g(y)$ still has to satisfy Assumptions~\ref{asmp:f}-\ref{asmp:A}.


Given the finite-sum structure \eqref{eqn:finite_sum}, we denote the individual gradient of each $L_i$ as
\begin{align*}
B_i\left(x,y\right) :=
\begin{bmatrix}
\nabla_xL_i(x, y)  \\
\nabla_yL_i(x, y) 
\end{bmatrix} =
\begin{bmatrix}
\nabla f_i(x) + A_i^\top y\\
A_ix - \nabla g_i\left(y\right)
\end{bmatrix}, 
\end{align*} and the full gradient of $L$ as
\begin{align*}
B\left(x,y\right) := 
\frac1n \sum_{i=1}^n B_i(x, y)=
\begin{bmatrix}
\frac{1}{n}\sum_{i=1}^{n}\left(\nabla f_i(x) + A_i^\top y\right)\\
\frac{1}{n}\sum_{i=1}^{n}\left(A_ix - \nabla g_i\left(y\right)\right)
\end{bmatrix}.
\end{align*}
A naive computation of $A_ix$ or $A_i^\top y$ takes $O(d_1d_2)$ time.
However, in many applications like policy evaluation~\citep{du2017stochastic} and empirical risk minimization, each $A_i$ is given as the outer product of two vectors (i.e., a rank-$1$ matrix), which makes $A_ix$ and $A_i^\top y$ computable in only $O(d)$ time, where $d = \max\left\{d_1,d_2\right\}$.
In this case, computing an individual gradient $B_i(x,y)$ takes $O(d)$ time while computing the full gradient $B(x,y)$ takes $O(nd)$ time.

We adapt the stochastic variance reduced gradient (SVRG) method~\citep{johnson2013accelerating} to solve Problem~\eqref{eqn:finite_sum}.
The algorithm uses two layers of loops.
In an outer loop, the algorithm first computes a full gradient using a ``snapshot'' point $(\tilde{x},\tilde{y})$, and
then the algorithm executes $N$ inner loops, where $N$ is a parameter to be chosen. In each inner loop, the algorithm randomly samples an index $i$ from $[n]$ and updates the current iterate $(x,y)$ using a variance-reduced stochastic gradient: 
\begin{align}
	B_{i}(x,y,\tilde{x},\tilde{y}) = B_{i}(x,y) + B(\tilde{x},\tilde{y}) - B_{i}(\tilde{x},\tilde{y}).
	\label{Equ:SVRG:ReducedVarGrad}
	\end{align}
Here, $B_{i}(x,y)$ is the stochastic gradient at $(x,y)$ 
computed using the random index $i$, and
$B(\tilde{x},\tilde{y}) - B_{i}(\tilde{x},\tilde{y})$ is a term
used to reduce the variance in $B_{i}(x,y)$ while keeping
$B_{i}(x,y,\tilde{x},\tilde{y})$ 
an unbiased estimate of $B(x,y)$. 
The full details of the algorithm are provided in Algorithm~\ref{algo:SVRG_saddle}.
For clarity, we denote by $(\tilde x_t, \tilde y_t)$ the snapshot point in the $t$-th epoch (outer loop), and denote by $(x_{t, 0}, y_{t, 0}), (x_{t, 1}, y_{t, 1}), \ldots$ all the intermediate iterates within this epoch.

The following theorem establishes the linear convergence guarantee of Algorithm~\ref{algo:SVRG_saddle}.

\begin{algorithm}[tb]
	\renewcommand{\algorithmicrequire}{\textbf{Inputs:}}
	\renewcommand{\algorithmicensure}{\textbf{Outputs:}}
	\caption{Primal-Dual SVRG}
	\label{algo:SVRG_saddle}
	\begin{algorithmic}[1]
		\REQUIRE initial points $\tilde x_0 \in \R^{d_1}, \tilde y_0\in \R^{d_2}$, 
		step sizes $\eta_1, \eta_2>0$,
		number of inner iterations $N \in \mathbb N$
		\FOR{$t = 0, 1, \ldots$}
		\STATE Compute $B(\tilde x_t, \tilde y_t)$
		\STATE  $(x_{t, 0}, y_{t, 0}) = (\tilde x_t, \tilde y_t)$
		\FOR{$j=0$ {\bfseries to} $N-1$}
		\STATE Sample an index $i_j$ uniformly from $[n]$
		\STATE Compute $B_{i_j}(x_{t, j},y_{t, j})$ and $B_{i_j}(\tilde x_t, \tilde y_t)$
		\STATE 
		 \quad $
		\begin{bmatrix}
		x_{t,j+1} \\
		y_{t,j+1}
		\end{bmatrix} $\\
		$ = \begin{bmatrix}
		x_{t,j} \\
		y_{t,j}
		\end{bmatrix}    -  \begin{bmatrix}
		\eta_1 I_{d_1} & 0 \\
		0 & -\eta_2 I_{d_2}
		\end{bmatrix}   
		B_{i_j}(x_{t,j},y_{t,j},\tilde x_t,\tilde y_t),
		$ \\[0.5ex]
		where $B_{i_j}(x_{t,j},y_{t,j},\tilde x_t,\tilde y_t)$
		is defined in~\eqref{Equ:SVRG:ReducedVarGrad}
		\ENDFOR
		\STATE  $(\tilde x_{t+1},\tilde y_{t+1}) = (x_{t, j_t},y_{t,j_t})$, where $j_t$ is an index sampled uniformly from $\left\{ 0, 1, \ldots, N-1 \right\}$
		\ENDFOR
	\end{algorithmic}
\end{algorithm}

\begin{thm}\label{thm:svrg}
There exists a choice of parameters $\eta_1, \eta_2 = \poly\left( \beta, \rho, M, 1/\alpha, 1/\sigma_{\min}(A) \right)^{-1}$ and
$N = \poly\left( \beta, \rho, M, 1/\alpha, 1/\sigma_{\min}(A) \right)$ in Algorithm~\ref{algo:SVRG_saddle},
as well as another number  $\mu  = \poly\left( \beta, \rho, M, 1/\alpha, 1/\sigma_{\min}(A) \right)$,
such that if we define $Q_t = \expect{ \norm{\tilde x_{t} - x^*}^2 + \mu \norm{\tilde y_{t} - \nabla g^*(A\tilde x_{t})}^2 }$, then Algorithm~\ref{algo:SVRG_saddle} guarantees $Q_{t+1} \le \frac12 Q_t$ for all $t$.
\end{thm}

Since computing a full gradient takes $O(nd)$ time and each inner loop takes $O(d)$ time,
each epoch takes $O(nd+Nd)$ time in total. 
Therefore, the total running time of Algorithm~\ref{algo:SVRG_saddle} is $O\left((n+N)d \log \frac1\epsilon \right)$
in order to reach an $\epsilon$-close solution, which is the desired running time of SVRG (note that $N$ does not depend on $n$).

The proof of Theorem~\ref{thm:svrg} is given in Appendix~\ref{sec:svrg_proof}. It relies on the same proof idea in Section~\ref{sec:pdbg} as well as the standard analysis technique for SVRG by~\cite{johnson2013accelerating}.

\section{Preliminary Empirical Evaluation}
\label{sec:exp}
\begin{figure*}[t!]
	\centering
	\begin{subfigure}[t]{0.3\textwidth}
		\includegraphics[width=\textwidth]{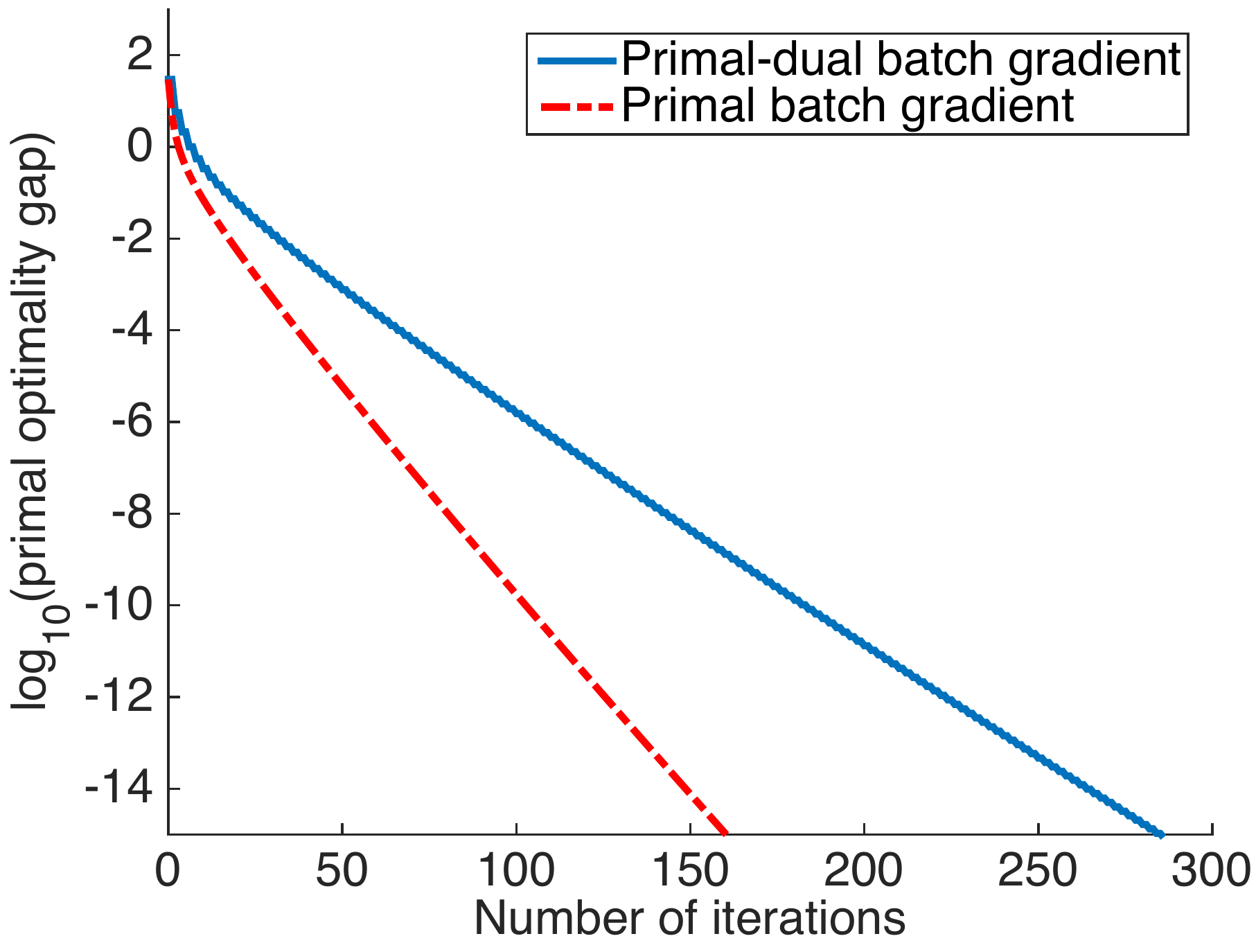}
		\caption{Data $\sim \cN(0, I_d)$}
	\end{subfigure}	
	\begin{subfigure}[t]{0.3\textwidth}
		\includegraphics[width=\textwidth]{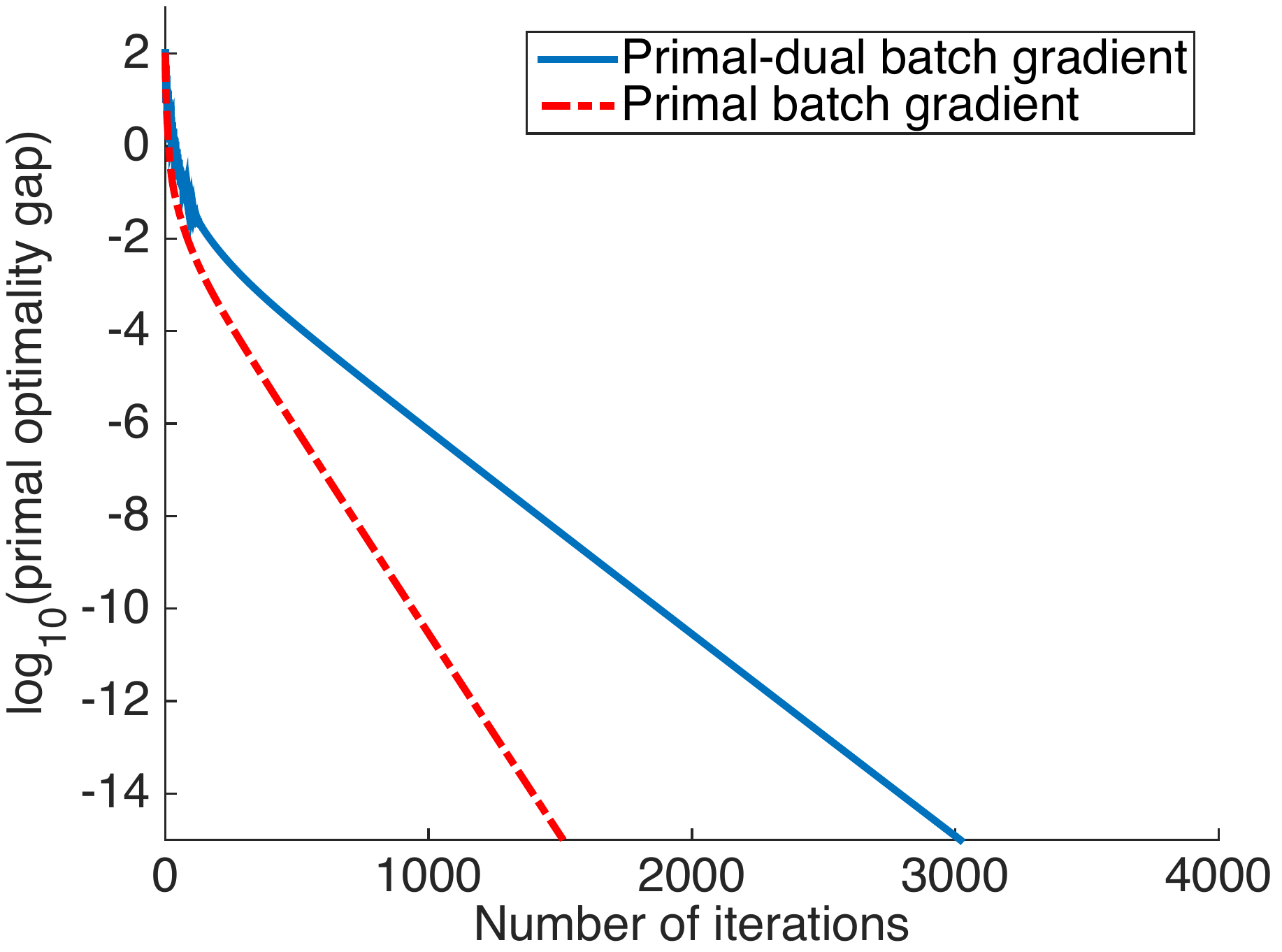}
		\caption{Data $\sim \cN(0, \Sigma)$, $\Sigma_{ij} = 2^{-|i-j|/2}$}
	\end{subfigure}
	\begin{subfigure}[t]{0.3\textwidth}
		\includegraphics[width=\textwidth]{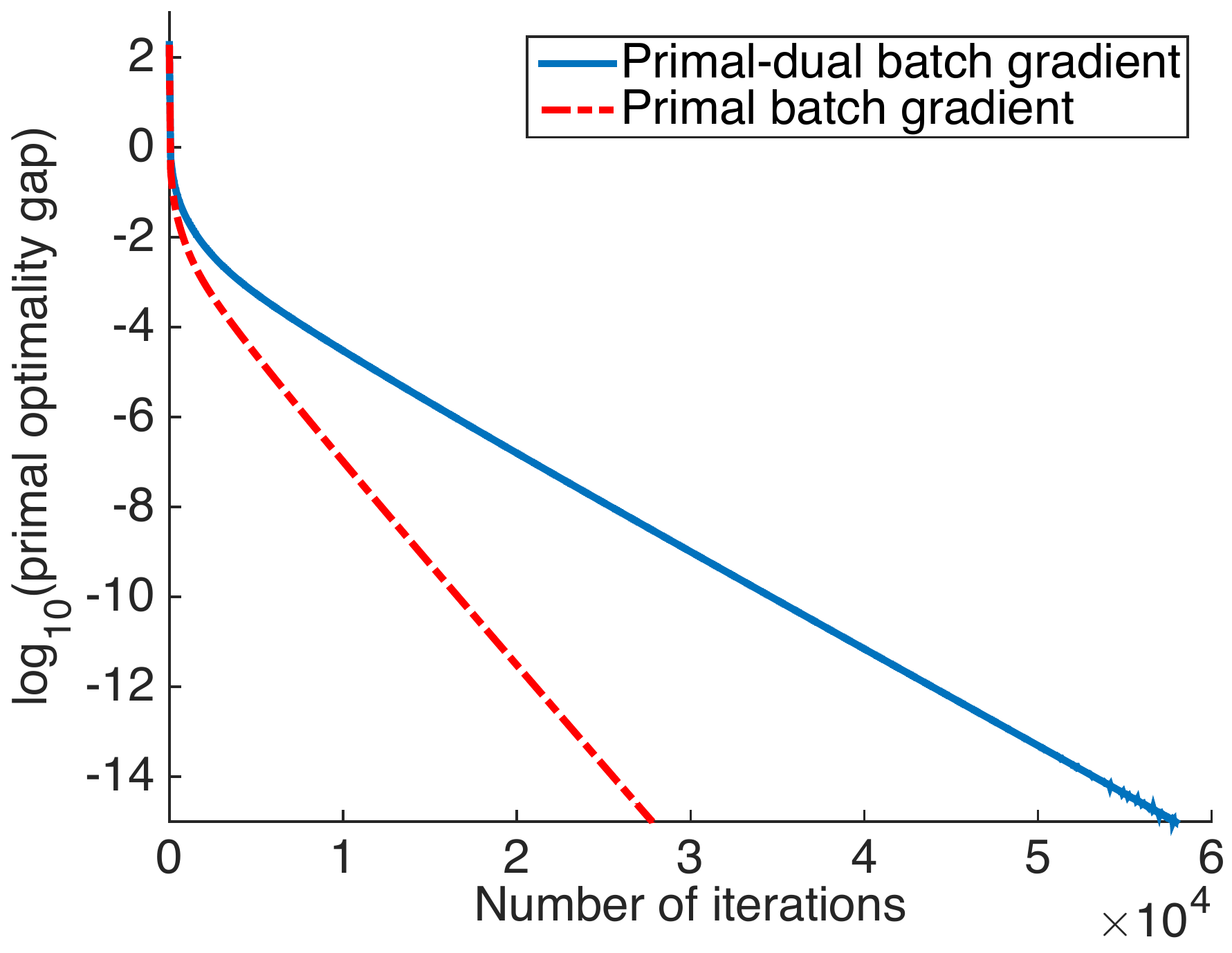}
		\caption{Data $\sim \cN(0, \Sigma)$, $\Sigma_{ij} = 2^{-|i-j|/10}$}
	\end{subfigure}
	\caption{Comparison of batch gradient methods for smoothed-$L_1$-regularized regression with $d = 200, n=500$.
	}
	\label{fig:batch}
\end{figure*}

\begin{figure*}[t!]
	\centering
	\begin{subfigure}[t]{0.3\textwidth}
		\includegraphics[width=\textwidth]{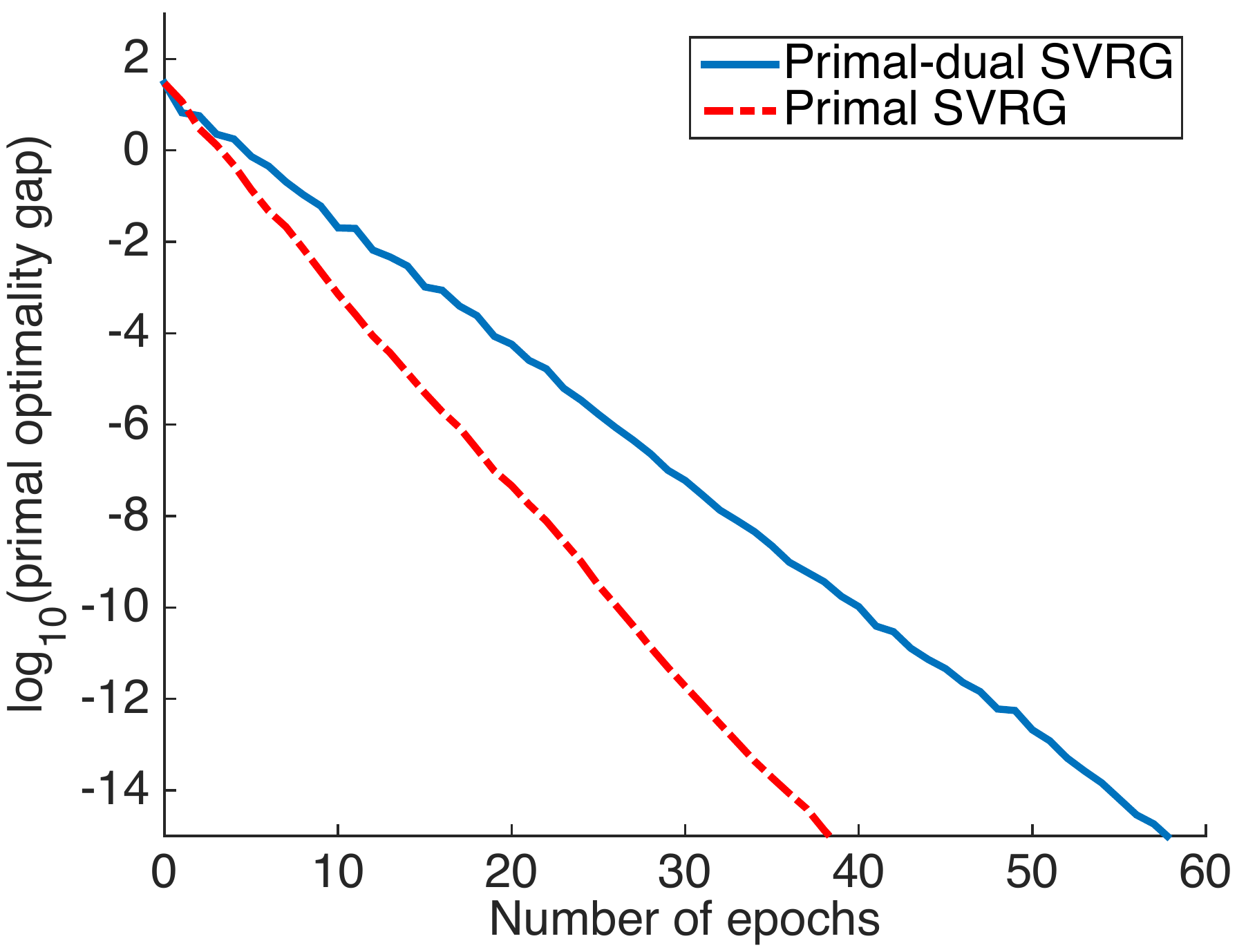}
		\caption{Data $\sim \cN(0, I_d)$}
	\end{subfigure}	
	\begin{subfigure}[t]{0.3\textwidth}
		\includegraphics[width=\textwidth]{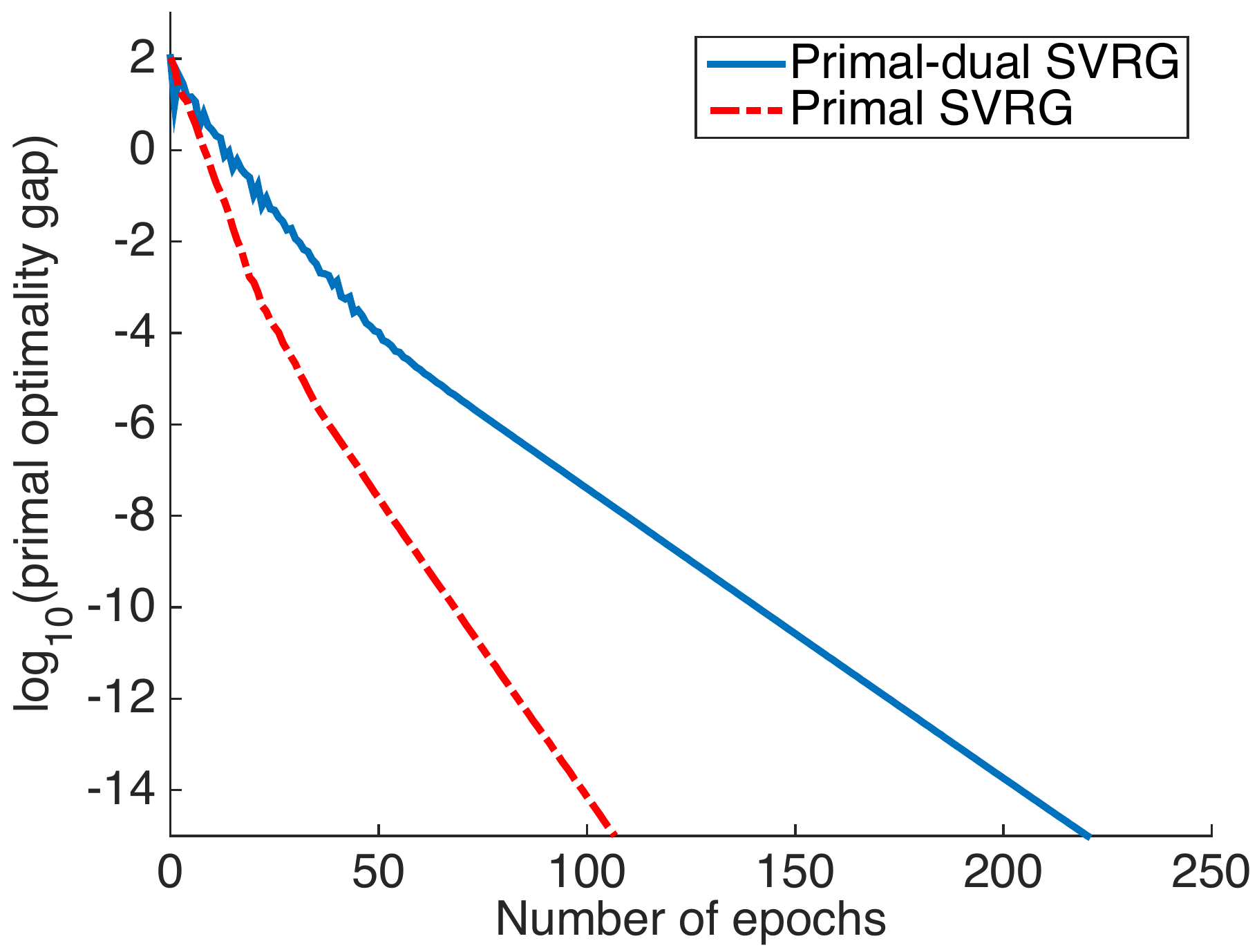}
		\caption{Data $\sim \cN(0, \Sigma)$, $\Sigma_{ij} = 2^{-|i-j|/2}$}
	\end{subfigure}
	\begin{subfigure}[t]{0.3\textwidth}
		\includegraphics[width=\textwidth]{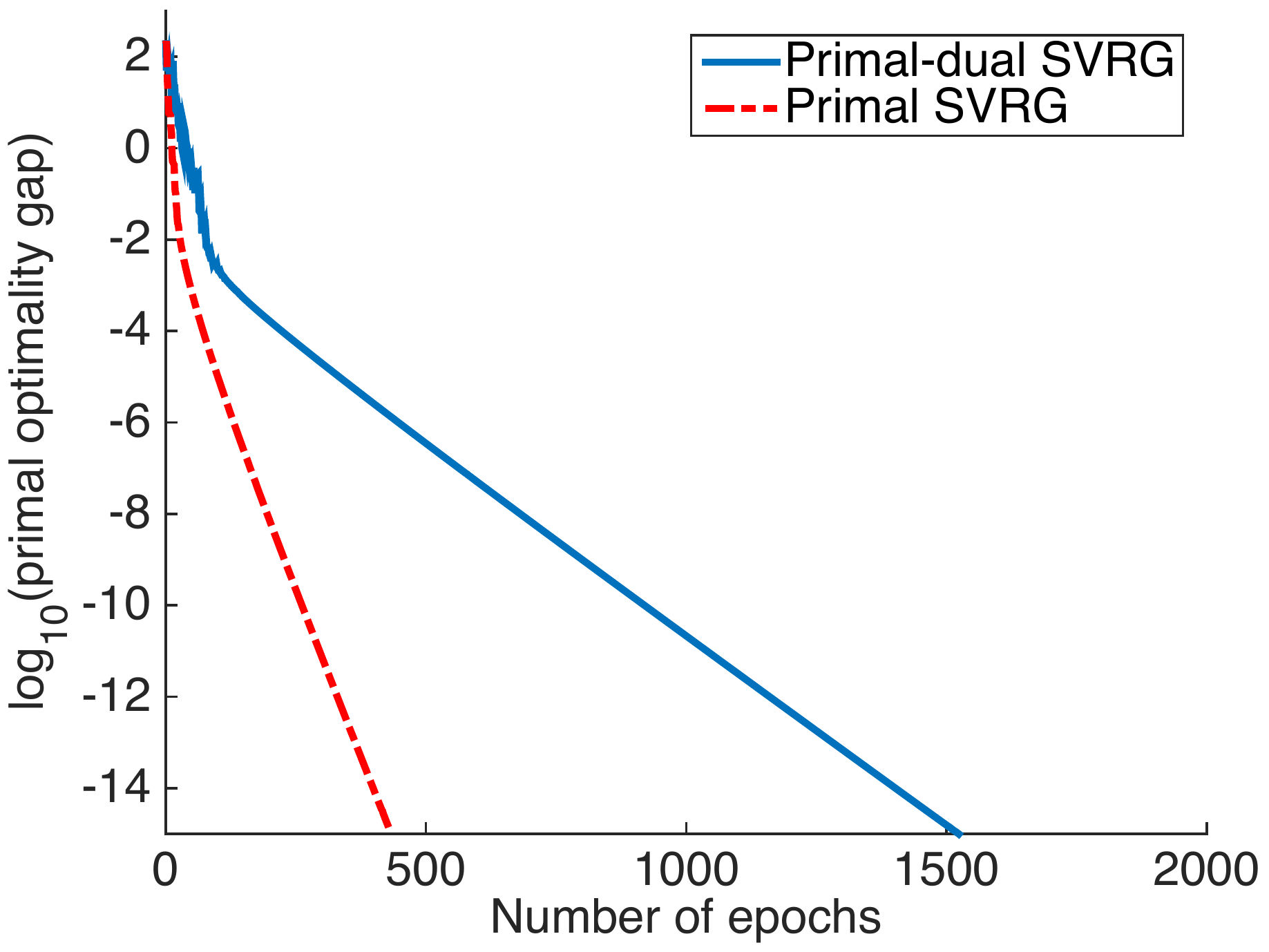}
		\caption{Data $\sim \cN(0, \Sigma)$, $\Sigma_{ij} = 2^{-|i-j|/10}$}
	\end{subfigure}
	\caption{Comparison of SVRG methods for smoothed-$L_1$-regularized regression with $d = 200$ and $n = 500$.
	}
	\label{fig:svrg-1}
\end{figure*}


We perform preliminary empirical evaluation for the following purposes:
(i) to verify that both the primal-dual gradient method (Algorithm~\ref{algo:pdg}) and the primal-dual SVRG method (Algorithm~\ref{algo:SVRG_saddle}) can indeed achieve linear convergence,
(ii) to investigate the convergence rates of Algorithms \ref{algo:pdg} and \ref{algo:SVRG_saddle}, in comparison with their primal-only counterparts (i.e., the usual gradient descent and SVRG algorithms for the primal problem),
and (iii) to compare the convergence rates of Algorithms~\ref{algo:pdg} and \ref{algo:SVRG_saddle}.

We consider the linear regression problem with smoothed-$L_1$ regularization, formulated as
\begin{equation} \label{eqn:smoothL1-regression}
\min_{x\in \R^d} \frac{1}{2n}\norm{Ax-b}^2 + \lambda R_a(x),
\end{equation}
where $A \in \R^{n\times d}$, $b\in \R^n$ , and $R_a(x) := \sum_{i=1}^d \frac1a \left( \log(1+e^{ax_i}) + \log(1+e^{-ax_i}) \right)$ is the \emph{smoothed -$L_1$ regularization} \citep{schmidt2007fast}.\footnote{When $a>0$ is large we have $R_a(x) \approx \|x\|_1$ for all $x\in \R^d$.}
Note that $R_a(x)$ is smooth but not strongly convex, and does not have a closed-form proximal mapping.
As discussed in Section~\ref{sec:examples},
Problem~\eqref{eqn:smoothL1-regression} admits a saddle point formulation:
$$\min_{x\in \R^d} \max_{y\in \R^n} \left\{ \frac1n\left(-\frac12\norm{y}^2 - b^\top y + y^\top Ax \right) + \lambda R_a(x) \right\}.$$
In this experiment we choose $a = 10$ and $\lambda = 0.01/n$.

We generate data (i.e. rows of $A$) from a Gaussian distribution $\cN(0, \Sigma)$, where we consider three cases: (a) $\Sigma = I_d$, (b) $\Sigma_{ij} = 2^{-|i-j|/2}$, and (c) $\Sigma_{ij} = 2^{-|i-j|/10}$.
These three choices result in small, medium, and large condition numbers of $A$, respectively.\footnote{The condition number of a matrix $A$ is defined as $\sigma_{\max}(A)/\sigma_{\min}(A)$.}
In Figures~\ref{fig:batch} and \ref{fig:svrg-1}, we plot the performances of batch gradient and SVRG algorithms, where we choose $d = 200$ and $n = 500$.
We tune the step sizes in every case 
in order to observe the optimal convergence rates.

These plots show that: (i) both Algorithm~\ref{algo:pdg} and Algorithm~\ref{algo:SVRG_saddle} can indeed achieve linear convergence, verifying our theorems; 
(ii) in all our examples, primal-dual methods always converge slower than the corresponding primal methods, but they are only slower by no more than 3 times;
(iii) 
Algorithm~\ref{algo:SVRG_saddle} has a much faster convergence rate than Algorithm~\ref{algo:pdg}, especially when the condition number is large, which verifies the theoretical result that SVRG can significantly reduce the computational complexity.


\section{Conclusion}
\label{sec:con}
We prove that the vanilla primal-dual gradient method can achieve linear convergence for 
 convex-concave saddle point 
problem~\eqref{eqn:basic_primal_dual} without strong convexity in the primal variable.
We develop a novel proof strategy  and further use this proof strategy to show the linear convergence of the primal-dual SVRG method for saddle point problems with finite-sum structures.
It would be interesting to study whether our technique can be used to analyze non-convex problems.



\subsubsection*{Acknowledgements}
\label{sec:ack}
We thank Qi Lei, Yuanzhi Li, Jialei Wang, Yining Wang and Lin Xiao for helpful discussions.
Wei Hu was supported by NSF, ONR, Simons Foundation, Schmidt Foundation, Mozilla Research, Amazon Research, DARPA and SRC.

\bibliography{ref}
\bibliographystyle{plainnat}
\onecolumn
\newpage
\appendix
\section*{\Large Appendix}
\section{Omitted Proofs}
\label{sec:proof}

\subsection{Finishing the Proof of Theorem~\ref{thm:main}} \label{sec:proof-main}

\begin{proof}[Proof of Theorem~\ref{thm:main}]
	Combining~\eqref{eqn:main-ineq-1} and~\eqref{eqn:main-ineq-2}, we obtain
	\begin{align*} 
	&P_{t+1} = \lambda a_{t+1} + b_{t+1}  \\
	\le& \left( 1-   \frac{\sigma_{\min}^2}{\beta} \eta_1 \right) \lambda a_t + \lambda \sigma_{\max} \eta_1 b_t  + \frac{\sigma_{\max}}{\alpha} \left( \rho + \frac{\sigma_{\max}^2}{\alpha} \right) \eta_1 a_t 
	+ \left( 1-\alpha \eta_2 + \frac{\sigma_{\max}^2}{\alpha}\eta_1 \right) b_t \\
	=& \left( 1- \frac{\sigma_{\min}^2}{\beta} \eta_1 + \frac1\lambda \frac{\sigma_{\max}}{\alpha} \left( \rho + \frac{\sigma_{\max}^2}{\alpha} \right) \eta_1 \right) \lambda a_t +   \left( 1-\alpha \eta_2 + \left( \frac{\sigma_{\max}^2}{\alpha} + \lambda \sigma_{\max} \right) \eta_1  \right) b_t.
	\end{align*}
	If we can choose $\lambda$, $\eta_1$ and $\eta_2$ such that both coefficients
	$$c_1 \coloneqq 1- \frac{\sigma_{\min}^2}{\beta} \eta_1 + \frac1\lambda \frac{\sigma_{\max}}{\alpha} \left( \rho + \frac{\sigma_{\max}^2}{\alpha} \right) \eta_1$$
	and $$c_2 \coloneqq 1-\alpha \eta_2 + \left( \frac{\sigma_{\max}^2}{\alpha} + \lambda\sigma_{\max} \right) \eta_1$$ are strictly less than $1$, we have linear convergence.

	It remains to show that our choices of parameters 
	\begin{align*}
	\lambda &= \frac{2 \beta \sigma_{\max}\cdot \left(\rho+\frac{\sigma_{\max}^2}{\alpha}\right)}{\alpha \sigma_{\min}^2},\\
	\eta_1 &= \frac{\alpha}{(\alpha+\beta) \left( \frac{\sigma_{\max}^2}{\alpha} + \lambda\sigma_{\max} \right)}, \\
	\eta_2 &= \frac{2}{\alpha+\beta},
	\end{align*}
	give the desired upper bound on $\max\{c_1, c_2\}$.
	
	First we verify that our choices of $\eta_1$ and $\eta_2$ satisfy the requirements in Propositions~\ref{prop:primal} and \ref{prop:dual}.
	It is clear that $\eta_2 \le \frac{2}{\alpha+\beta}$ is satisfied.
	For $\eta_1$, we have
	\begin{align*}
	\eta_1
	\le \frac{\alpha}{\beta  \lambda\sigma_{\max} }
	= \frac{\alpha^2 \sigma_{\min}^2}{2\beta^2 \sigma_{\max}^2 \cdot \left(\rho+\frac{\sigma_{\max}^2}{\alpha}\right) }
	\le \frac{1}{2 \left(\rho+\frac{\sigma_{\max}^2}{\alpha}\right)}
	\le \frac{1}{\rho + \frac{\sigma_{\max}^2}{\alpha} + \frac{\sigma_{\min}^2}{\beta}}.
	\end{align*}
	Therefore the requirements in Propositions~\ref{prop:primal} and \ref{prop:dual} are satisfied.
	
	Next we calculate $c_1$ and $c_2$.
	Since $\left( \frac{\sigma_{\max}^2}{\alpha} + \lambda\sigma_{\max} \right) \eta_1 = \left( \frac{\sigma_{\max}^2}{\alpha} + \lambda\sigma_{\max} \right) \frac{\alpha}{(\alpha+\beta) \left( \frac{\sigma_{\max}^2}{\alpha} + \lambda\sigma_{\max} \right)} = \frac{\alpha}{\alpha+\beta} = \frac{\alpha}{2}\eta_2$, we have
	\begin{equation} \label{eqn:c2}
	c_2 = 1 - \frac12 \alpha \eta_2 = 1 - \frac{\alpha}{\alpha+\beta}.
	\end{equation}
	For $c_1$, since
	$ \frac1\lambda \frac{\sigma_{\max}}{\alpha} \left( \rho + \frac{\sigma_{\max}^2}{\alpha} \right)
	= \frac{\alpha \sigma_{\min}^2}{2 \beta \sigma_{\max}\cdot \left(\rho+\frac{\sigma_{\max}^2}{\alpha}\right)} \cdot \frac{\sigma_{\max}}{\alpha} \left( \rho + \frac{\sigma_{\max}^2}{\alpha} \right)
	= \frac{ \sigma_{\min}^2}{2 \beta } $, we have
	\begin{equation} \label{eqn:in-proof-1}
	c_1 = 1 - \frac{\sigma_{\min}^2}{2\beta} \eta_1
	= 1 - \frac{\sigma_{\min}^2}{2\beta} \cdot \frac{\alpha}{(\alpha+\beta) \left( \frac{\sigma_{\max}^2}{\alpha} + \lambda\sigma_{\max} \right)}.
	\end{equation}
	Note that $\lambda \sigma_{\max} = \frac{2 \beta \sigma_{\max}^2 \cdot \left(\rho+\frac{\sigma_{\max}^2}{\alpha}\right)}{\alpha \sigma_{\min}^2} \ge \frac{2\beta \sigma_{\max}^4}{\alpha^2 \sigma_{\min}^2} \ge \frac{2\sigma_{\max}^2}{\alpha}$.
	Then \eqref{eqn:in-proof-1} implies
	\begin{equation} \label{eqn:c1}
	\begin{aligned}
	c_1 &\le 1 -  \frac{\alpha\sigma_{\min}^2}{2\beta(\alpha+\beta) \left( \frac12 \lambda\sigma_{\max} + \lambda\sigma_{\max} \right)}
	=  1 -  \frac{\alpha\sigma_{\min}^2}{3\beta(\alpha+\beta)  \lambda\sigma_{\max} } \\
	&= 1 -  \frac{\alpha^2 \sigma_{\min}^4}{6\beta^2 (\alpha+\beta)  \sigma_{\max}^2 \cdot \left(\rho+\frac{\sigma_{\max}^2}{\alpha}\right)  }
	\le 1 -  \frac{\alpha^2 \sigma_{\min}^4}{12\beta^3  \sigma_{\max}^2 \cdot \left(\rho+\frac{\sigma_{\max}^2}{\alpha}\right)  }.
	\end{aligned}
	\end{equation}
	Combining \eqref{eqn:c2} and \eqref{eqn:c1}, we obtain
	\begin{align*}
	\max\{c_1, c_2\} \le \max\left\{ 1 - \frac{\alpha}{\alpha+\beta}, 1 -  \frac{\alpha^2 \sigma_{\min}^4}{12\beta^3  \sigma_{\max}^2 \cdot \left(\rho+\frac{\sigma_{\max}^2}{\alpha}\right)  } \right\}
	= 1 -  \frac{\alpha^2 \sigma_{\min}^4}{12\beta^3  \sigma_{\max}^2 \cdot \left(\rho+\frac{\sigma_{\max}^2}{\alpha}\right)  }.
	\end{align*}
	Therefore the proof is completed.
\end{proof}

\subsection{Proof of Theorem~\ref{thm:svrg}}
\label{sec:svrg_proof}

\begin{proof}[Proof of Theorem~\ref{thm:svrg}]
Denote $\sigma_{\min} := \sigma_{\min}(A)$.
Let 
\begin{align*}
\varphi_i(x, y) &:= \nabla_xL_i(x, y)  = \nabla f_i(x) + A_i^\top y, \\
\varphi(x, y) &:= \nabla_xL(x, y)  = \nabla f(x) + A^\top y, \\
\psi_i(x, y) &:= \nabla_yL_i(x, y)  =  A_i x - \nabla g_i(y), \\
\psi(x, y) &:= \nabla_yL(x, y)  =  Ax - \nabla g(y),
\end{align*}
and define $\theta: \R^{d_1} \to \R^{d_2}$ as $\theta(x) := \nabla g^*(Ax)$.
Note that we have $\theta(x^*) = y^*$ from \eqref{eqn:opt-condition}.

\paragraph{Step 1: Bound for the Primal Variable.}
Consider one iteration of the inner loop $x_{t, j+1} = x_{t, j} - \eta_1 \left(  \varphi_{i_j}(x_{t, j}, y_{t, j}) + \varphi(\tilde x_{t}, \tilde y_{t}) - \varphi_{i_j}(\tilde x_{t}, \tilde y_{t}) \right)$.
We now only consider the randomness in $i_j$, conditioned on everything in previous iterations.
Because we have $\E[x_{t, j+1}] = x_{t, j} - \eta_1  \varphi(x_{t, j}, y_{t, j}) $, using the equation $\E\norm{\xi}^2 = \norm{\E \xi}^2 + \E \norm{\xi-\E\xi}^2$
we have:
\begin{equation} \label{eqn:svrg-primal-bound}
\begin{aligned}
\E{\norm{x_{t,j+1}-x^*}^2} = 
& \norm{\expect{x_{t,j+1}-x^*}}^2 + \expect{\norm{\left(x_{t,j+1}-x^*\right)-\expect{x_{t,j+1}-x^*}}^2} \\
= & \norm{x_{t, j} - \eta_1  \varphi(x_{t, j}, y_{t, j}) - x^*}^2 + \eta_1^2 \expect{\norm{ \varphi_{i_j}(x_{t, j}, y_{t, j}) + \varphi(\tilde x_{t}, \tilde y_{t}) - \varphi_{i_j}(\tilde x_{t}, \tilde y_{t}) - \varphi(x_{t, j}, y_{t, j}) }^2}.
\end{aligned}
\end{equation}
For the first term in~\eqref{eqn:svrg-primal-bound}, we note that it has the same form as the update rule in Algorithm~\ref{algo:pdg}
for the primal variable.
Therefore, we can apply Proposition~\ref{prop:primal} and get (noticing $\sigma_{\max}(A)\le M$, and assuming $\eta_1$ is sufficiently small)
\begin{equation} \label{eqn:svrg-primal-part1}
\begin{aligned}
&\norm{x_{t,j} - \eta_1 \varphi(x_{t, j}, y_{t, j}) -x^*}^2\\
\le& \left(    \left( 1-   \frac{\sigma_{\min}^2}{\beta}\eta_1 \right) \left\| x_{t, j} - x^* \right\| + M \eta_1 \left\|  y_{t, j} - \theta(x_{t, j})  \right\|    \right)^2 \\
\le& \left[ \left( 1-   \frac{\sigma_{\min}^2}{\beta}\eta_1 \right) + \frac{\sigma_{\min}^2}{\beta}\eta_1 \right] \cdot \left[    \left( 1-   \frac{\sigma_{\min}^2}{\beta}\eta_1 \right) \left\| x_{t, j} - x^* \right\|^2 + \frac{\beta M^2}{\sigma_{\min}^2} \eta_1 \left\|  y_{t, j} - \theta(x_{t, j})  \right\|^2    \right] \\
=&  \left( 1-   \frac{\sigma_{\min}^2}{\beta}\eta_1 \right) \left\| x_{t, j} - x^* \right\|^2 + \frac{\beta M^2}{\sigma_{\min}^2} \eta_1 \left\|  y_{t, j} - \theta(x_{t, j}) \right\|^2 ,
\end{aligned}
\end{equation}
where the second inequality is due to Cauchy-Schwarz inequality.

Next, we bound the second term in \eqref{eqn:svrg-primal-bound}.
First, using $\norm{a+b}^2 \le 2\|a\|^2+2\|b\|^2$ we get
\begin{align*}
&\eta_1^2 \expect{\norm{ \varphi_{i_j}(x_{t, j}, y_{t, j}) + \varphi(\tilde x_{t}, \tilde y_{t}) - \varphi_{i_j}(\tilde x_{t}, \tilde y_{t}) - \varphi(x_{t, j}, y_{t, j}) }^2} \\
\le\,& 2\eta_1^2 \expect{\norm{ \varphi_{i_j}(x_{t, j}, y_{t, j})  - \varphi(x_{t, j}, y_{t, j}) - \varphi_{i_j}(x^*, y^*) + \varphi(x^*, y^*)  }^2}
\\& + 2\eta_1^2 \expect{\norm{  \varphi(\tilde x_{t}, \tilde y_{t}) - \varphi_{i_j}(\tilde x_{t}, \tilde y_{t}) +  \varphi_{i_j}(x^*, y^*) - \varphi(x^*, y^*) }^2} .
\end{align*}
Note that $\expect{\varphi_{i_j}(x_{t, j}, y_{t, j})  - \varphi_{i_j}(x^*, y^*) } = \varphi(x_{t, j}, y_{t, j})  - \varphi(x^*, y^*) $ and
$\expect{\varphi_{i_j}(\tilde x_{t}, \tilde y_{t})  - \varphi_{i_j}(x^*, y^*) } = \varphi(\tilde x_{t}, \tilde y_{t})  - \varphi(x^*, y^*) $.
Using $\E \norm{\xi-\E\xi}^2 \le \E\norm{\xi}^2$ we get
\begin{align*}
&\eta_1^2 \expect{\norm{ \varphi_{i_j}(x_{t, j}, y_{t, j}) + \varphi(\tilde x_{t}, \tilde y_{t}) - \varphi_{i_j}(\tilde x_{t}, \tilde y_{t}) - \varphi(x_{t, j}, y_{t, j}) }^2} \\
\le\,& 2\eta_1^2 \expect{\norm{ \varphi_{i_j}(x_{t, j}, y_{t, j})  - \varphi_{i_j}(x^*, y^*)   }^2+ \norm{ \varphi_{i_j}(\tilde x_{t}, \tilde y_{t}) - \varphi_{i_j}(x^*, y^*)  }^2} .
\end{align*}
Then from 
\begin{equation} \label{eqn:varphi-smoothness}
\begin{aligned}
\norm{ \varphi_{i}(x, y)  - \varphi_i(x^*, y^*)   }
&= \norm{ \nabla f_i(x) + A_i^\top y - \nabla f_i(x^*) - A_i^\top y^* } \\
&\le \norm{ \nabla f_i(x) - \nabla f_i(x^*) } + \norm{  A_i^\top y  - A_i^\top y^* } \\
&\le \rho \norm{x-x^*} + M \norm{y-y^*} \\
&\le \rho \norm{x-x^*} + M \left( \norm{y- \theta(x)} + \norm{\theta(x) - y^*}  \right) \\
&= \rho \norm{x-x^*} + M\left( \norm{y-\theta(x)} + \norm{\theta(x) - \theta(x^*)}  \right) \\
&\le \rho \norm{x-x^*} + M\left( \norm{y-\theta(x)} + \frac{M}{\alpha}\norm{x-x^*}  \right) \\
&= \left( \rho +M^2/\alpha  \right) \norm{x-x^*} + M  \norm{y-\theta(x)}
\end{aligned}
\end{equation}
we obtain
\begin{equation} \label{eqn:svrg-primal-part2}
\begin{aligned}
&\eta_1^2 \expect{\norm{ \varphi_{i_j}(x_{t, j}, y_{t, j}) + \varphi(\tilde x_{t}, \tilde y_{t}) - \varphi_{i_j}(\tilde x_{t}, \tilde y_{t}) - \varphi(x_{t, j}, y_{t, j}) }^2} \\
\le\,& 2\eta_1^2  \left[ \left( \rho +M^2/\alpha  \right) \norm{x_{t, j}-x^*} + M \norm{y_{t, j}- \theta(x_{t, j})} \right]^2 \\ &+2\eta_1^2 \left[ \left( \rho +M^2/\alpha  \right) \norm{\tilde x_{t}-x^*} + M \norm{\tilde y_{t}- \theta(\tilde x_{t})} \right]^2 .
\end{aligned}
\end{equation}

Plugging \eqref{eqn:svrg-primal-part1} and \eqref{eqn:svrg-primal-part2} into \eqref{eqn:svrg-primal-bound}, we have
\begin{equation} \label{eqn:svrg-primal-final}
\begin{aligned}
& \E{\norm{x_{t,j+1}-x^*}^2} \\
\le & \left( 1-   \frac{\sigma_{\min}^2}{\beta}\eta_1 \right) \left\| x_{t, j} - x^* \right\|^2 + \frac{\beta M^2}{\sigma_{\min}^2} \eta_1 \norm{ y_{t, j} - \theta(x_{t, j}) }^2 \\
&+2\eta_1^2  \left[ \left( \rho +M^2/\alpha  \right) \norm{x_{t, j}-x^*} + M \norm{y_{t, j}- \theta(x_{t, j})} \right]^2 \\ 
&+2\eta_1^2 \left[ \left( \rho +M^2/\alpha  \right) \norm{\tilde x_{t}-x^*} + M  \norm{\tilde y_{t}- \theta(\tilde x_{t})} \right]^2 \\
\le& \left( 1 - \eta_1/c_1 \right) \norm{x_{t, j}-x^*}^2 + c_2 \eta_1 \norm{ y_{t, j} - \theta(x_{t, j}) }^2 +c_3\eta_1^2 \norm{\tilde x_{t}-x^*}^2 + c_4\eta_1^2 \norm{\tilde y_{t}- \theta(\tilde x_{t})}^2,
\end{aligned}
\end{equation}
where $c_1, \ldots, c_4$ all have the form $\poly\left( \beta, \rho, M, 1/\alpha, 1/\sigma_{\min} \right)$.
Here we assume $\eta_1$ is sufficiently small.

\paragraph{Step 2: Bound for the Dual Variable.}
The dual update takes the form $y_{t, j+1} = y_{t, j} + \eta_2 \left(  \psi_{i_j}(x_{t, j}, y_{t, j}) + \psi(\tilde x_{t}, \tilde y_{t}) - \psi_{i_j}(\tilde x_{t}, \tilde y_{t}) \right)$.
Same as before,
We first only consider the randomness in $i_j$, conditioned on everything in previous iterations.

We have
\begin{equation} \label{eqn:svrg-dual-bound}
\begin{aligned}
&\expect{\norm{y_{t,j+1} - \theta(x_{t,j+1})}^2} \\
\le\,& \expect{  \left( \norm{y_{t,j+1} - \theta(x_{t,j})} + \norm{\theta(x_{t,j}) - \theta(x_{t,j+1})}   \right)^2 } \\
=\,& \expect{   \norm{y_{t,j+1} - \theta(x_{t,j})}^2 + \norm{\theta(x_{t,j}) - \theta(x_{t,j+1})}^2 + 2 \norm{y_{t,j+1} - \theta(x_{t,j})} \cdot \norm{\theta(x_{t,j}) - \theta(x_{t,j+1})}    } \\
\le\,&  \E \norm{y_{t,j+1} - \theta(x_{t,j})}^2 + \E \norm{\theta(x_{t,j}) - \theta(x_{t,j+1})}^2 + 2 \sqrt{ \E \norm{y_{t,j+1} - \theta(x_{t,j})}^2 \cdot \E \norm{\theta(x_{t,j}) - \theta(x_{t,j+1})}^2  }   \\
=\,& A + B + 2\sqrt{AB},
\end{aligned}
\end{equation}
where $A := \E \norm{y_{t,j+1} - \theta(x_{t,j})}^2 $ and $B := \E \norm{\theta(x_{t,j}) - \theta(x_{t,j+1})}^2$.
The second inequality above is due to Cauchy-Schwarz inequality.
Thus it remains to bound $A$ and $B$.

We can bound $A$ similar to \eqref{eqn:svrg-primal-bound}:
\begin{equation} \label{eqn:svrg-dual-A}
\begin{aligned}
A& = \E{\norm{y_{t,j+1}-\theta(x_{t, j})}^2} \\
&=  \norm{\expect{x_{t,j+1}-\theta(x_{t, j})}}^2 + \expect{\norm{\left(x_{t,j+1}-\theta(x_{t, j})\right)-\expect{x_{t,j+1}-\theta(x_{t, j})}}^2} \\
&=  \norm{y_{t, j} + \eta_2  \psi(x_{t, j}, y_{t, j}) - \theta(x_{t, j})}^2 + \eta_2^2 \expect{\norm{ \psi_{i_j}(x_{t, j}, y_{t, j}) + \psi(\tilde x_{t}, \tilde y_{t}) - \psi_{i_j}(\tilde x_{t}, \tilde y_{t}) - \psi(x_{t, j}, y_{t, j}) }^2}.
\end{aligned}
\end{equation}
For the first term in \eqref{eqn:svrg-dual-A}, we can directly apply \eqref{eqn:dual-decrease-1} and get (assuming $\eta_2$ to be sufficiently small)
\begin{equation} \label{eqn:svrg-dual-A-part1}
\norm{y_{t, j} + \eta_2  \psi(x_{t, j}, y_{t, j}) - \theta(x_{t, j})}
\le (1-\alpha\eta_2) \norm{y_{t, j} - \theta(x_{t, j})}.
\end{equation}
The second term in \eqref{eqn:svrg-dual-A} can be bounded in the same way as we did for the second term in \eqref{eqn:svrg-primal-bound}:
\begin{equation} \label{eqn:svrg-dual-A-part2}
\begin{aligned}
&\eta_2^2 \expect{\norm{ \psi_{i_j}(x_{t, j}, y_{t, j}) + \psi(\tilde x_{t}, \tilde y_{t}) - \psi_{i_j}(\tilde x_{t}, \tilde y_{t}) - \psi(x_{t, j}, y_{t, j}) }^2} \\
\le\,& 2\eta_2^2 \expect{\norm{ \psi_{i_j}(x_{t, j}, y_{t, j})  - \psi(x_{t, j}, y_{t, j}) - \psi_{i_j}(x^*, y^*) + \psi(x^*, y^*)  }^2}
\\& + 2\eta_2^2 \expect{\norm{  \psi(\tilde x_{t}, \tilde y_{t}) - \psi_{i_j}(\tilde x_{t}, \tilde y_{t}) +  \psi_{i_j}(x^*, y^*) - \psi(x^*, y^*) }^2} \\
\le\,& 2\eta_2^2 \expect{\norm{ \psi_{i_j}(x_{t, j}, y_{t, j})  - \psi_{i_j}(x^*, y^*)  }^2} + 2\eta_2^2 \expect{\norm{    \psi_{i_j}(\tilde x_{t}, \tilde y_{t}) -  \psi_{i_j}(x^*, y^*)  }^2} \\
\le\,& 2\eta_2^2 \left( M  \left(1+\beta/\alpha  \right) \norm{x_{t, j}-x^*} + \beta\norm{y_{t, j}-\theta(x_{t, j})}  \right)^2 \\&+ 2\eta_2^2 \left( M  \left(1+\beta/\alpha  \right) \norm{\tilde x_{t}-x^*} + \beta\norm{\tilde y_{t}-\theta(\tilde x_{t})}  \right)^2,
\end{aligned}
\end{equation}
where we have used
\begin{align*}
\norm{ \psi_{i}(x, y)  - \psi_i(x^*, y^*)   }
&= \norm{ A_i x - \nabla g_i(y) -  A_i x^* + \nabla g_i(y^*)  } \\
&\le \norm{  A_ix  - A_ix^* } + \norm{ \nabla g_i(y) - \nabla g_i(y^*) }  \\
&\le M \norm{x-x^*} + \beta \norm{y-y^*} \\
&\le M \norm{x-x^*} + \beta \left( \norm{y- \theta(x)} + \norm{\theta(x) - y^*}  \right) \\
&= M\norm{x-x^*} + \beta \left( \norm{y-\theta(x)} + \norm{\theta(x) - \theta(x^*)}  \right) \\
&\le M \norm{x-x^*} + \beta \left( \norm{y-\theta(x)} + \frac{M}{\alpha}\norm{x-x^*}  \right) \\
&=  M \left(1+\beta/\alpha  \right) \norm{x-x^*} + \beta\norm{y-\theta(x)} .
\end{align*}
Plugging \eqref{eqn:svrg-dual-A-part1} and \eqref{eqn:svrg-dual-A-part2} into \eqref{eqn:svrg-dual-A} we get
\begin{equation} \label{eqn:svrg-dual-A-bound}
\begin{aligned}
A \le\, & (1-\alpha\eta_2)^2 \norm{y_{t, j} - \theta(x_{t, j})}^2 + 2\eta_2^2 \left( M  \left(1+\beta/\alpha  \right) \norm{x_{t, j}-x^*} + \beta\norm{y_{t, j}-\theta(x_{t, j})}  \right)^2 \\&+ 2\eta_2^2 \left( M  \left(1+\beta/\alpha  \right) \norm{\tilde x_{t}-x^*} + \beta\norm{\tilde y_{t}-\theta(\tilde x_{t})}  \right)^2 \\
\le\,& \left( 1 - \eta_2/c_5 \right) \norm{y_{t, j}-\theta(x_{t, j})}^2 + c_6 \eta_2^2 \norm{ x_{t, j} - x^* }^2 +c_7\eta_2^2 \norm{\tilde x_{t}-x^*}^2 + c_8\eta_2^2 \norm{\tilde y_{t}- \theta(\tilde x_{t})}^2,
\end{aligned}
\end{equation}
where $c_5, \ldots, c_8$ all have the form $\poly\left( \beta, M, 1/\alpha \right)$.
Here we assume $\eta_1$ is sufficiently small.

For $B$, we have
\begin{equation*}
\begin{aligned}
B &= \E \norm{\theta(x_{t,j}) - \theta(x_{t,j+1})}^2 \\
&\le \left( M/\alpha \right)^2 \E \norm{x_{t, j+1} - x_{t, j}}^2 \\
&= \left( M/\alpha \right)^2 \eta_1^2 \E \norm{ \varphi_{i_j}(x_{t, j}, y_{t, j}) + \varphi(\tilde x_{t}, \tilde y_{t}) - \varphi_{i_j}(\tilde x_{t}, \tilde y_{t}) }^2 \\
&\le 2\left( M/\alpha \right)^2 \eta_1^2 \left( \E \norm{ \varphi_{i_j}(x_{t, j}, y_{t, j}) + \varphi(\tilde x_{t}, \tilde y_{t}) - \varphi_{i_j}(\tilde x_{t}, \tilde y_{t}) - \varphi(x_{t, j}, y_{t, j}) }^2 + \E \norm{\varphi(x_{t, j}, y_{t, j})}^2 \right) \\
&= 2\left( M/\alpha \right)^2 \eta_1^2 \left( \E \norm{ \varphi_{i_j}(x_{t, j}, y_{t, j}) + \varphi(\tilde x_{t}, \tilde y_{t}) - \varphi_{i_j}(\tilde x_{t}, \tilde y_{t}) - \varphi(x_{t, j}, y_{t, j}) }^2 + \E \norm{\varphi(x_{t, j}, y_{t, j}) - \varphi(x^*, y^*)}^2 \right).
\end{aligned}
\end{equation*}
Then using \eqref{eqn:svrg-primal-part2} and the smoothness of $\varphi$ (\eqref{eqn:varphi-smoothness} holds for $\varphi$ as well) we obtain
\begin{equation} \label{eqn:svrg-dual-B-bound}
\begin{aligned}
B \le\, & 4\left( M/\alpha \right)^2 \eta_1^2  \left[ \left( \rho +M^2/\alpha  \right) \norm{x_{t, j}-x^*} + M  \norm{y_{t, j}- \theta(x_{t, j})} \right]^2  \\
&+ 4\left( M/\alpha \right)^2 \eta_1^2 \left[ \left( \rho +M^2/\alpha  \right) \norm{\tilde x_{t}-x^*} + M\norm{\tilde y_{t}- \theta(\tilde x_{t})} \right]^2 \\
&+ 2\left( M/\alpha \right)^2 \eta_1^2  \left[ \left( \rho +M^2/\alpha  \right) \norm{x_{t, j}-x^*} + M  \norm{y_{t, j}- \theta(x_{t, j})} \right]^2 \\
=\,& 6\left( M/\alpha \right)^2 \eta_1^2  \left[ \left( \rho +M^2/\alpha  \right) \norm{x_{t, j}-x^*} + M  \norm{y_{t, j}- \theta(x_{t, j})} \right]^2  \\
&+ 4\left( M/\alpha \right)^2 \eta_1^2 \left[ \left( \rho +M^2/\alpha  \right) \norm{\tilde x_{t}-x^*} + M  \norm{\tilde y_{t}- \theta(\tilde x_{t})} \right]^2 \\
\le\,& c_9 \eta_1^2 \norm{y_{t, j}-\theta(x_{t, j})}^2 + c_{10} \eta_1^2 \norm{ x_{t, j} - x^* }^2 +c_{11}\eta_1^2 \norm{\tilde x_{t}-x^*}^2 + c_{12}\eta_1^2 \norm{\tilde y_{t}- \theta(\tilde x_{t})}^2,
\end{aligned}
\end{equation}
where $c_9, \ldots, c_{12}$ all have the form $\poly\left( \beta, M, 1/\alpha \right)$.

Therefore, plugging \eqref{eqn:svrg-dual-A-bound} and \eqref{eqn:svrg-dual-B-bound} into \eqref{eqn:svrg-dual-bound}, we get
\begin{equation}  \label{eqn:svrg-dual-final}
\begin{aligned}
& \expect{\norm{y_{t,j+1} - \theta(x_{t,j+1})}^2} \\
\le\,& A+B+2\sqrt{AB} \\
\le\,& A+B + \eta_1 A + \frac{B}{\eta_1} \\
\le\,& (1+\eta_1) \left( \left( 1 - \eta_2/c_5 \right) \norm{y_{t, j}-\theta(x_{t, j})}^2 + c_6 \eta_2^2 \norm{ x_{t, j} - x^* }^2 +c_7\eta_2^2 \norm{\tilde x_{t}-x^*}^2 + c_8\eta_2^2 \norm{\tilde y_{t}- \theta(\tilde x_{t})}^2 \right) \\
&+ (1+1/\eta_1) \left( c_9 \eta_1^2 \norm{y_{t, j}-\theta(x_{t, j})}^2 + c_{10} \eta_1^2 \norm{ x_{t, j} - x^* }^2 +c_{11}\eta_1^2 \norm{\tilde x_{t}-x^*}^2 + c_{12}\eta_1^2 \norm{\tilde y_{t}- \theta(\tilde x_{t})}^2 \right)\\
\le\, & \left( 1 - \eta_2/c_{13} \right) \norm{y_{t, j}-\theta(x_{t, j})}^2 + c_{14} \eta_2^2 \norm{ x_{t, j} - x^* }^2 + c_{15}\eta_2^2 \norm{\tilde x_{t}-x^*}^2 + c_{16}\eta_2^2 \norm{\tilde y_{t}- \theta(\tilde x_{t})}^2,
\end{aligned}
\end{equation}
where $c_{13}, \ldots, c_{16}$ all have the form $\poly\left( \beta, M, 1/\alpha \right)$.
Here we assume $\eta_1$ is chosen sufficiently small \emph{given $\eta_2$}.

\paragraph{Step3: Putting Things Together.}
Let $p_t := \E \norm{\tilde x_t - x^*}^2$ and $q_t := \E \norm{\tilde y_t - \theta(\tilde x_t)}^2$.
Taking expectation with respect to everything we were conditioned on, we can write \eqref{eqn:svrg-primal-final} as
\begin{equation*} 
\begin{aligned}
\E{\norm{x_{t,j+1}-x^*}^2} 
\le \left( 1 - \eta_1/c_1 \right) \E\norm{x_{t, j}-x^*}^2 + c_2 \eta_1 \E \norm{ y_{t, j} - \theta(x_{t, j}) }^2 +c_3\eta_1^2 p_t + c_4\eta_1^2 q_t.
\end{aligned}
\end{equation*}
Taking sum over $j=0, 1, \ldots, N-1$, and noticing that $ \tilde x_t = x_{t, 0}$ and that $\tilde x_{t+1} = x_{t, j_t}$ for a random $j_t \in \{0, 1, \ldots, N-1\}$, we obtain
\begin{equation*} 
\begin{aligned}
\frac1N \left( \E{\norm{x_{t,N}-x^*}^2} - \E{\norm{\tilde x_{t}-x^*}^2} \right)
\le  - \frac{\eta_1}{c_1} \E\norm{\tilde x_{t+1}-x^*}^2 + c_2 \eta_1  \E\norm{\tilde y_{t+1} - \theta(x_{t+1}) }^2 +c_3\eta_1^2 p_t + c_4\eta_1^2 q_t,
\end{aligned}
\end{equation*}
which implies
\begin{equation*} 
\begin{aligned}
 -\frac1N p_t
\le  - \frac{\eta_1}{c_1} p_{t+1}  + c_2 \eta_1  q_{t+1} +c_3\eta_1^2 p_t + c_4\eta_1^2 q_t,
\end{aligned}
\end{equation*}
i.e.
\begin{equation*}
p_{t+1} \le \left( \frac{c_1}{\eta_1N} + c_1c_3\eta_1 \right) p_t + c_1 c_2 q_{t+1} + c_1c_4\eta_1q_t.
\end{equation*}
Similarly, from \eqref{eqn:svrg-dual-final} we can get
\begin{equation*}
q_{t+1} \le \left( \frac{c_{13}}{\eta_2N} + c_{13}c_{16}\eta_2 \right) q_t + c_{13} c_{14} \eta_2 p_{t+1} + c_{13}c_{15}\eta_2 p_t.
\end{equation*}
Then it is easy to see that one can choose $\mu$ and $N $ to be sufficiently large $\poly\left( \beta, \rho, M, 1/\alpha, 1/\sigma_{\min} \right)$ and choose $\eta_1$ and $\eta_2$ to be sufficiently small $\poly\left( \beta, \rho, M, 1/\alpha, 1/\sigma_{\min} \right)^{-1}$ such that
\begin{equation*}
p_{t+1} + \mu q_{t+1} \le \frac12 \left( p_t + \mu q_t \right).
\end{equation*}
This completes the proof of Theorem~\ref{thm:svrg}.
\end{proof}

\section{Linear Convergence of the Primal-Dual Gradient Method When Both $f$ and $g$ are Smooth and Strongly Convex}
\label{sec:both_strong_proof}
In this section we show that if both $f$ and $g$ are smooth and strongly convex, Algorithm~\ref{algo:pdg} can achieve linear convergence for Problem~\eqref{eqn:basic_primal_dual}.
Note that this proof is much simpler than that of Theorem~\ref{thm:main}.

We denote $\sigma_{\max} := \sigma_{\max}(A)$.

\begin{thm}
	Suppose $f$ is $\beta_1$-smooth and $\alpha_1$-strongly convex, and $g$ is $\beta_2$-smooth and $\alpha_2$-strongly convex.
	If we choose $\eta_1 = \min\left\{ \frac{1}{\alpha_1+\beta_1}, \frac{\alpha_2}{4\sigma_{\max}^2} \right\}$ and $\eta_2 = \min\left\{ \frac{1}{\alpha_2+\beta_2}, \frac{\alpha_1}{4\sigma_{\max}^2} \right\}$ in Algorithm~\ref{algo:pdg}
	and let $R_t = \eta_2 \norm{x_t - x^*}^2 + \eta_1 \norm{y_t - y^*}^2$, then we have
	\begin{equation*}
	R_{t+1} \le \left( 1 - \frac12 \min\left\{ \frac{\alpha_1}{\alpha_1+\beta_1}, \frac{\alpha_2}{\alpha_2+\beta_2}, \frac{\alpha_1\alpha_2}{4\sigma_{\max}^2} \right\} \right) R_t.
	\end{equation*}
\end{thm}
\begin{proof}
From the update rule $x_{t+1} = x_t - \eta_1 \nabla_x L(x_t, y_t) $ we have
\begin{equation} \label{eqn:2sc-inproof}
\norm{x_{t+1}-x^*}^2 = \norm{x_t - x^*}^2 - 2\eta_1 \left\langle x_t - x^*, \nabla_x L(x_t, y_t) \right\rangle + \eta_1^2 \norm{\nabla_x L(x_t, y_t)}^2.
\end{equation}
The inner product term in \eqref{eqn:2sc-inproof} can be bounded as
\begin{align*}
\left\langle x_t - x^*,\nabla_x L(x_t, y_t) \right\rangle 
&= \left\langle x_t - x^*, \nabla_x L(x_t, y_t) - \nabla_x L(x^*, y^*) \right\rangle \\
&= \langle x_t - x^*, \nabla f(x_t) - \nabla f(x^*)\rangle + \langle x_t -x ^*, A^\top(y_t-y^*)\rangle \\
&\ge  \frac{\alpha_1\beta_1}{\alpha_1+\beta_1}\norm{x_t - x^*}^2 + \frac{1}{\alpha_1 + \beta_1}\norm{\nabla f(x_t) - \nabla f(x^*)}^2 + \langle x_t -x ^*, A^\top(y_t-y^*)\rangle,
\end{align*}
where we have used Lemma 3.11 in \citep{bubeck2015convex}.

The third term in \eqref{eqn:2sc-inproof} can be bounded as
\begin{align*}
\norm{\nabla_x L(x_t, y_t)}^2
&= \norm{\nabla_x L(x_t, y_t) - \nabla_x L(x^*, y^*)}^2
\\&= \norm{\nabla f(x_t) - \nabla f(x^*) + A^\top(y_t - y^*)}^2
\\&\le 2\left( \norm{\nabla f(x_t) - \nabla f(x^*)}^2 + \norm{A^\top(y_t - y^*)}^2 \right)
\\&\le 2 \left(\norm{\nabla f(x_t)-\nabla f(x^*)}^2 + \sigma_{\max}^2\norm{y_t - y^*}^2\right).
\end{align*}

By symmetry, for the dual variable, we have an inequality similar to \eqref{eqn:2sc-inproof}.
Combining everything together, and using $\eta_1 \le \frac{1}{\alpha_1 + \beta_1}$ and $\eta_2 \le \frac{1}{\alpha_2+\beta_2}$, with some routine calculations we can get
 \begin{align*}
&\eta_2 \norm{x_{t+1}-x^*}^2 + \eta_1 \norm{y_{t+1}-y^*}^2 \\
\le &\left(1-2\alpha_1\eta_1+2\alpha_1^2\eta_1^2+2\sigma_{\max}^2\eta_1\eta_2\right)\eta_2\norm{x_t-x^*}^2 + \left(1-2\alpha_2\eta_2 +2\alpha_2^2\eta_2^2 + 2\sigma_{\max}^2\eta_1\eta_2\right)\eta_1\norm{y_t-y^*}^2.
\end{align*}
Then, from our choices of $\eta_1$ and $\eta_2$, the above inequality implies
\begin{align*}
R_{t+1}
&\le \left(1-2\alpha_1\eta_1+\alpha_1\eta_1+ \frac12 \alpha_1\eta_1\right)\eta_2\norm{x_t-x^*}^2 + \left(1-2\alpha_2\eta_2 +\alpha_2\eta_2 + \frac12\alpha_2\eta_2\right)\eta_1\norm{y_t-y^*}^2 \\
&\le \left(1 - \frac{1}{2}\min\{\alpha_1\eta_1, \alpha_2\eta_2 \}\right) R_t \\
&= \left( 1 - \frac12 \min\left\{ \frac{\alpha_1}{\alpha_1+\beta_1}, \frac{\alpha_2}{\alpha_2+\beta_2}, \frac{\alpha_1\alpha_2}{4\sigma_{\max}^2} \right\} \right) R_t. \qedhere
\end{align*}
\end{proof}

\end{document}